\documentclass[11pt]{article}

\usepackage{amsmath,amssymb,xfrac}
\usepackage{algorithm,algorithmic}
\usepackage{url,hyperref,cleveref}
\usepackage[margin=1in]{geometry}
\usepackage{authblk}
\usepackage{pbox,tabularx,tabulary}
\usepackage[table]{xcolor}
\definecolor{gray}{cmyk}{0,0,0,0.25}
\usepackage{tikz}
\usetikzlibrary{shapes.geometric,arrows.meta}
\usepackage[square,numbers]{natbib}
\usepackage{amsthm}
\newtheorem{theorem}{Theorem}
\newtheorem{proposition}{Proposition}

\newtheorem*{theorem*}{Theorem}
\newtheorem*{proposition*}{Proposition}
\newtheorem*{lemma*}{Lemma}
\newtheorem*{corollary*}{Corollary}
\newtheorem{assumption}{Assumption}
\theoremstyle{definition}
\newtheorem{definition}{Definition}

\newcommand{\mbb}{\mathbb}

\newcommand{\defn}{\equiv}
\newcommand{\st}{\mathrm{s.t.}\;}

\newcommand{\group}[1]{\left( #1 \right)}
\newcommand{\abs}[1]{\left| #1 \right|}

\newcommand{\set}[1]{\left\{ #1 \right\}}
\newcommand{\seq}[2][]{\left( #2 \right)_{#1}}

\newcommand{\card}[1]{\left| #1 \right|}

\renewcommand{\hat}{\widehat}
\renewcommand{\tilde}{\widetilde}
\renewcommand{\emptyset}{\varnothing}
\DeclareMathOperator{\smallsum}{\textstyle{\sum}}

\newcommand{\allg}{{g}}

\newcommand{\nashset}{F}

\makeatletter
\def\munderbar#1{\underline{\sbox\tw@{$#1$}\dp\tw@\z@\box\tw@}}
\makeatother

\def\thetitle{Equilibrium modeling and solution approaches inspired by nonconvex bilevel programming}

\title{\thetitle%
\thanks{This preprint has not undergone peer review or any post-submission improvements or corrections.
The Version of Record of this article is published in {\it Computational Optimization and Applications}, and is available online at \url{https://doi.org/10.1007/s10589-023-00524-w}.
}
}

\author[1]{Stuart Harwood}
\author[2]{Francisco Trespalacios}
\author[1]{Dimitri Papageorgiou}
\author[2]{Kevin Furman}

\affil[1]{ExxonMobil Corporate Strategic Research, Annandale, NJ 08801, USA}
\affil[2]{ExxonMobil Upstream Research Company, Spring, TX 77389, USA}

\begin{document}
\maketitle

\begin{abstract}
Methods for finding pure Nash equilibria have been dominated by variational inequalities and complementarity problems.
Since these approaches fundamentally rely on the sufficiency of first-order optimality conditions for the players' decision problems, they only apply as heuristic methods when the players are modeled by nonconvex optimization problems.
In contrast, this work approaches Nash equilibrium using theory and methods for the global optimization of nonconvex bilevel programs.
Through this perspective, we draw precise connections between Nash equilibria, feasibility for bilevel programming, the Nikaido-Isoda function, and classic arguments involving Lagrangian duality and spatial price equilibrium.
Significantly, this is all in a general setting without the assumption of convexity.
Along the way, we introduce the idea of \emph{minimum disequilibrium} as a solution concept that reduces to traditional equilibrium when an equilibrium exists.
The connections with bilevel programming and related semi-infinite programming permit us to adapt global optimization methods for those classes of problems, such as constraint generation or cutting plane methods, to the problem of finding a minimum disequilibrium solution.
We propose a specific algorithm and show that this method can find a pure Nash equilibrium even when the players are modeled by mixed-integer programs.
Our computational examples include practical applications like unit commitment in electricity markets.
\end{abstract}


\section{Introduction}

In game theory, pure Nash equilibrium (PNE) has a long history.
Recent work includes applications in electricity markets \citep{fullerEA17,gabriel2013B,guo2021copositive}, energy markets \citep{gabriel_etal,gabriel2013A}, competitive capacitated lot-sizing \citep{li2011competition}, and applications of integer programming games \citep{carvalho2017nash,dragotto2021zero}.
This particular body of work has been pushing the limits of the theory and methods for PNE by considering players modeled by nonconvex optimization problems.
We note that PNE, as the name suggests, is a solution in ``pure'' strategies;
this in contrast with Nash's original definition \citep{nash50}, which deals with equilibrium in mixed strategies
(roughly, a solution consisting of a probability distribution over the strategy set of each player).
Nonconvexity is a major challenge when trying to find a pure Nash equilibrium.

A popular class of methods for finding a pure Nash equilibrium involves taking the KKT conditions for each player problem and solving a complementarity problem, or more generally a variational inequality (VI).
See \cite{facchinei_pang,facchinei_pang2010} for reviews.
However, this approach fundamentally relies on the sufficiency of first-order optimality conditions for the players' decision problems.
When the players' decision problems are defined by smooth functions, there are regularity conditions under which solving a VI reformulation can give a Nash equilibrium for a related game \cite{pang2011nonconvex}, but in general, without convex structure, this approach can only guarantee convergence to a ``local'' Nash equilibrium.

Consequently, when the players do not have convex structure, and in particular are discretely constrained, modeling with complementarity problems presents significant challenges.
Some authors try to adapt the complementarity-based methods, but the resulting methods are ultimately heuristic in nature.
For instance, the approaches of \cite{gabriel2013A,gabriel2013B} combine integer decisions with complementarity models to arrive at a mixed-integer complementarity problem.
The solution methods then relax certain integrality and complementarity conditions so that they obtain a mixed-integer linear program.
While this is a clear approach to a numerical implementation, the solution that one arrives at is not guaranteed to be an equilibrium, even if one exists.
In contrast, Fuller and \c{C}elebi~\citep{fullerEA17} introduce and motivate ``minimum total opportunity cost'' as a solution concept for unit commitment problems in electricity markets with nonconvexities arising from integer decisions.
However, the structure of the resulting problem led the authors to propose a heuristic solution method, and ultimately show that the ``minimum complementarity'' approach of \cite{gabriel2013A} is an approximate solution for their minimum total opportunity cost idea.

Approaches to finding equilibria that use the Nikaido-Isoda (NI) function typically reformulate the equilibrium problem as an optimization problem \citep{gurkan2009approximations,vonheusingerEA09}.
The NI function also appears in some methods that employ a fixed-point formulation to find an equilibrium \citep{contrerasEA04,uryasevEA94}.
While the basic optimization reformulations are often applicable when there is a certain amount of nonconvexity, the numerical methods so far proposed are limited to the setting where the player problems are convex, or at least have convex constraint sets.
This is in part because, in the most general case, an equilibrium is a global minimizer of an implicitly defined and nonconvex function;
finding such a point is a challenge.
Consequently, none of the approaches involving the NI function have so far considered discretely-constrained player problems.

Potential games \citep{monderer1996potential,ui2000shapley} have special structure that admit solution by a single (potentially nonconvex) optimization problem.
While this approach can be applicable without assuming convexity of the player problems, the necessary structure imposes strong conditions implying that a PNE must exist when the players' feasible sets are compact and a continuous potential exists \citep[Lemma 2.1]{monderer1996potential}.
In contrast, in this work we are interested in situations when a PNE is not known {\it a priori} to exist, and what alternatives to equilibrium are available when an equilibrium does not in fact exist.
Further, we will allow the possibility of global or side constraints that must be satisfied at equilibrium;
such a feature complicates reformulation as a potential game.

One of the ultimate goals of this paper is to describe a solution method for finding a pure Nash equilibrium that applies even when the player problems are nonconvex, including discretely-constrained problems.
This method and our overall perspective are inspired by the literature on global optimization for bilevel programs with nonconvex followers, and related semi-infinite programs \citep{bard83,guerravazquezEA08,mitsos11,steinEA02,tsoukalasEA09}.
Our fundamental view is that the problem of finding a PNE is equivalent to a bilevel feasibility problem;
through purely mathematical considerations for handling this feasibility problem, we introduce our relaxed solution concept called \emph{minimum disequilibrium}.
A minimum disequilibrium solution coincides with a PNE when a PNE exists;
otherwise, minimum disequilibrium provides the best alternative, in a particular sense, to an equilibrium.
Further, minimum disequilibrium has meaning in specific applications;
for instance, in unit commitment problems for electricity markets, our method finds the minimum total opportunity cost solution from \cite{fullerEA17}.

Subsequently, by specializing a global optimization method for semi-infinite programming to the minimum disequilibrium problem, we arrive at a very powerful and general procedure that may be applied even when a player is modeled by a mixed-integer nonlinear program (MINLP).
The method has characteristics of a cutting plane or constraint generation method, and relies on the solution of a series of related optimization problems to form a convergent sequence of upper and lower bounds on the measure of disequilibrium.
If these bounds converge to zero, we have an equilibrium;
otherwise, the lower bound will converge to a value strictly greater than zero, and we have a certificate that equilibrium does not exist.

We also discuss the specialization of the minimum disequilibrium problem to the economic situation of price-taking players.
In this setting, the minimum disequilibrium problem decomposes into a primal problem, often referred to as the maximum social welfare problem, and a dualized version of it.
Consequently, we show that optimal primal-dual solutions correspond with a minimum disequilibrium solution.
As a result, we are able to generalize the classic result that optimal dual variables of the maximum social welfare problem are equilibrium prices in the spatial price equilibrium problem \citep{samuelson1952spatial}.
Our generalization shows that this may hold even if the maximum social welfare problem is, for instance, an MINLP.

The present work has some elements in common with \cite{dragotto2021zero,guo2021copositive}.
These papers also aim to compute PNE for games involving discretely-constrained players.
Guo et al.~\citep{guo2021copositive} focus on players modeled by mixed-binary quadratic programs, and leverage a reformulation of these problems as completely positive programs;
consequently, they are able to use a strong duality result between completely positive programs and copositive programs \citep{burer2009copositive}.
They also consider pricing schemes for electricity markets.
Interestingly, since their approach is based on a strong duality result, they still fundamentally rely on a certain set of KKT conditions.
While this approach is effective for this class of problems, we feel this highlights the need for a different perspective, and that the present work addresses this need.
Meanwhile, Dragotto and Scatamacchia~\citep{dragotto2021zero} focus on enumerating all pure Nash equilibria for integer programming games, a class of games where each player's decision problem is an integer linear program.
They propose a cutting plane algorithm for characterizing the convex hull of the set of pure Nash equilibrium solutions.
By focusing on a specific class of games, they are able to specialize their results and, for instance, enumerate all equilibria.
In this work, we will aim to be more general, allowing players to have decision/feasible sets with an uncountable number of points.

To reiterate, the main contributions of this work include:
\begin{itemize}
	\item A numerical method for finding a PNE, if one exists, even if the players have general nonconvex decision problems.
	To the best of our knowledge, this is the first method that can provably find an arbitrarily accurate approximate PNE of a game with players defined by an MINLP.
	\item A rigorous method for finding the minimum total opportunity cost solution in unit commitment problems.
	\item A generalization of the result that optimal dual variables are equilibrium prices in spatial price equilibrium problems;
	this holds even if the corresponding maximum social welfare problem is a general nonconvex optimization problem.
\end{itemize}
\Cref{tab:methods} summarizes how our proposed method compares with previously published methods for finding PNE.
A feasible set described as ``mixed-integer, linearly constrained'' refers to a set that can be represented as
\[\set{(z^c,z^d) \in \mbb{R}^{n_c}\times \mbb{Z}^{n_d} : Az^c + Bz^d \le b}\]
for appropriate sized real matrices $A$, $B$, and real vector $b$.
A set described as ``integer, linearly constrained'' would be similar, but with $n_c = 0$
(no continuous variables).

\begin{table}
\caption{Methods for finding pure Nash equilibrium (PNE), when they apply, and guarantees}
\label{tab:methods}
\centering
\small
\rowcolors{1}{gray}{white}
\begin{tabulary}{\textwidth}{CCCC}
\hline
Method	
	& Players' objective functions
		& Players' feasible sets
			& Notes/guarantees
\\
\hline
\textbf{This work}: Reformulation as semi-infinite program, constraint generation method
	& \textbf{continuous} (encompasses everything below)
	& \textbf{compact} (encompasses everything below if bounded)
	& \textbf{finds PNE if one exists}, requires global solution of subproblems \\
Cutting plane method \citep{dragotto2021zero}
	& linear
	& integer, linearly constrained
	& finds PNE if one exists, requires global solution of subproblems \\
Copositive duality reformulation, cutting plane method \citep{guo2021copositive}
	& convex quadratic
	& mixed-integer, linearly constrained
	& finds PNE if one exists, requires global solution of mixed-integer linear programs \\
Minimum complementarity method \citep{gabriel2013A,gabriel2013B}
	& linear
	& mixed-integer, linearly constrained
	& may not find PNE even if one exists \\
Optimization reformulation, gradient-based minimization \citep{gurkan2009approximations,vonheusingerEA09}
	& continuously differentiable, convex
	& closed, convex
	& PNE exists, monotonicity assumptions required for convergence of gradient method \\
VI reformulation \citep[Prop.~1.4.2]{facchinei_pang}
	& continuously differentiable, convex
	& closed, convex
	& PNE exists, monotonicity assumptions required for convergence of methods for VI \\
\hline
\end{tabulary}
\end{table}

This work is organized as follows.
In \Cref{sec:defn_equiv}, we formally state the definition of the game and the corresponding notion of an equilibrium in pure strategies.
We also establish a few connections with other frameworks.
In \Cref{sec:characterization}, we proceed to characterize PNE as a solution to a bilevel programming feasibility problem and introduce the minimum disequilibrium formulation.
In \Cref{sec:solution_methods}, we leverage the clear connections with semi-infinite programming and propose our solution method for finding a minimum disequilibrium solution under very general conditions.
Since a minimum disequilibrium solution coincides with a pure Nash equilibrium when the latter exists, our method will find an equilibrium if one exists.
In \Cref{sec:pc_duality}, we specialize the discussion on minimum disequilibrium to a particular setting consistent with spatial price equilibrium.
In this setting, we establish that the minimum disequilibrium problem decomposes into a primal problem and its Lagrangian dual;
this permits a generalization to the nonconvex case of the classic result regarding equilibrium prices and dual variables.
Finally, in \Cref{sec:ex}, we provide some examples to illustrate the application of these solution methods.
These include an application to unit commitment problems, with integer-constrained players, and a spatial price equilibrium problem where the max social welfare problem is an MINLP.

\section{Definition of equilibrium and relationship to other frameworks}
\label{sec:defn_equiv}
In this section we present a form of a non-cooperative, complete information game with side constraints along with the associated definition of an equilibrium in pure strategies.
We proceed to show how this game can handle a more common/typical form of game, as well as how it is a special case of generalized games.

To motivate our alternate form of game and definition of equilibrium, consider two different settings for modeling economic equilibrium: price-taking players versus Cournot players.
The traditional informal definition of a PNE as a point from which each player has no incentive to deviate, given the decisions of the other players, is convenient when modeling Cournot players, since each player's objective depends directly on the decisions of the other players.
However, in the case of price-taking behavior, the players of the game only depend on quantities (prices) that no player explicitly determines.
This inspires our modified definition of the players and equilibrium.
The definition is in essence derived from a Stackelberg game, or more generally a single-leader/multi-follower game, in which the leader has a trivial (identically zero) objective.
This alternate definition and associated notation introduce a bilevel structure.
The result is a natural setting for our subsequent numerical developments, which are inspired by the literature on bilevel programming and related semi-infinite programming.

\subsection{Definition of constrained pure Nash equilibrium}
\label{sec:defn}
Let the $m$ players of the game be indexed by $i \in I \defn \set{1,\dots,m}$, and let each player be modeled by an optimization problem parametric in $x \in \mbb{R}^{n_0}$:
\begin{equation}
\label{player_i}
\tag{$\mathcal{A}_i$}
\begin{aligned}
S_i(x) \defn \arg
\min_{y_i}\; & g_i(x, y_i) \\
\st
& y_i \in Y_i,
\end{aligned}
\end{equation}
where
$Y_i \subset \mbb{R}^{n_i}$
and
$g_i : \mbb{R}^{n_0} \times Y_i \to \mbb{R}$.
This formulation for the player problems is quite general;
for instance, in some of the numerical examples we consider, each $Y_i$ enforces certain decision variables to be integer-valued.
We also introduce the global or side constraint set
\[
	G \subset \mbb{R}^{n_0} \times \mbb{R}^{n_1} \times \dots \times \mbb{R}^{n_m},
\]
which can be used to link the solutions of the player problems.
We denote this game as \\
$\mathcal{G}(G, g_1,Y_1, \dots, g_m,Y_m)$
or $\mathcal{G}(G, \seq[i\in I]{g_i,Y_i})$.
We have the following definition.

\begin{definition}
\label{defn:cpne}
A point $(x^*,y_1^*,\dots,y_m^*)$ is a \emph{constrained pure Nash equilibrium} (cPNE) of the game $\mathcal{G}(G, \seq[i\in I]{g_i,Y_i})$ if
$(x^*,y_1^*,\dots,y_m^*) \in G$ and $y_i^* \in S_i(x^*)$ for each $i \in I$.
\end{definition}

Another important definition is the optimal value function of problem~\eqref{player_i}
\begin{equation}
\label{eq:optimal_value}
g^*_i(x) \defn \inf_{y_i} \set{ g_i(x,y_i) : y_i \in Y_i }.
\end{equation}
As usual, for any $i$, define 
$g^*_i(x) = -\infty$ if optimization problem~\eqref{player_i} is unbounded
(and $g^*_i(x) = +\infty$ if it is infeasible, although for the most part we assume each $Y_i$ is nonempty).

As a brief note on notation, we will sometimes write
\[\begin{aligned}
&(y_1, \dots, y_m) = y \in \mbb{R}^{(\sum_{i \in I} n_i)}
\text{ and } \\
&(x,y_1, \dots, y_m) = (x,y) \in \mbb{R}^{n_0} \times  \mbb{R}^{(\sum_{i\in I} n_i)}.
\end{aligned}\]
More generally, a symbol without a subscript refers to a tuple/block vector of the subscripted objects;
for example we have $z = (z_1, z_2, \dots, z_J)$ for some vectors or scalars $z_j$.
Throughout this work we will not take for granted the existence of solutions to various mathematical optimization problems.
Further, we will often deal with the optimal value of an optimization problem.
This means that, for consistency if nothing else, we will denote the general optimization problem of 
``minimize the function $f$ on the set $X$'' 
as
``$\inf_x \set{f(x) : x \in X}$.''
As usual, ``$\arg \min$'' refers to the solution set, empty or not.

\subsection{Relationship to other frameworks and definitions}
\label{sec:others}
We briefly mention some other formalisms and definitions of equilibrium and how the game specified in \Cref{sec:defn} relates to them.

\subsubsection{Modeling a common form of game}
\label{sec:common_ref}
A more common form of game does not consider side constraints and makes the assumption that each player is parameterized by the decisions of all other players;
for each $i \in I$, player $i$ is modeled as
\begin{equation}
\notag
\begin{aligned}
S_i'(y_{-i}) = \arg
\min_{y_i}\; & g_i'(y_{-i}, y_i) \\
\st
& y_i \in Y_i,
\end{aligned}
\end{equation}
where
$y_{-i}$ denotes the vector of variables excluding the vector $y_i$:
\[
y_{-i} = (y_1, \dots, y_{i-1}, y_{i+1}, \dots, y_{m}).
\]
In this case we denote the game $\mathcal{G}'(\seq[i\in I]{g_i',Y_i})$.
A pure Nash equilibrium of $\mathcal{G}'(\seq[i\in I]{g_i',Y_i})$ is then defined as a point $(y_1^*, \dots, y_m^*)$ such that
$y_i^* \in S_i'(y_{-i}^*)$ for each $i \in I$
(see for instance \cite{facchinei_pang,facchinei_pang2010}).

Given the game $\mathcal{G}'(\seq[i\in I]{g_i',Y_i})$, we can convert it to a game in the form from \Cref{sec:defn} by simply specifying the side constraints so that the solution components of each player can be ``communicated'' to each other, essentially by defining the $x$ variables as a copy of the $y$ variables:
\[
G' = 
\set{ 
	(x_1, \dots, x_m, y_1,\dots, y_m) 
	\in \prod_i \mbb{R}^{n_i} \times \prod_i \mbb{R}^{n_i}
	: x_i = y_i, \forall i \in I 
}.
\]
If we define $g_i : (x,y_i) \mapsto g_i'(x_{-i}, y_i)$, then it is a simple exercise to check that
$y^*$ is a PNE of $\mathcal{G}'(\seq[i\in I]{g_i',Y_i})$ if and only if $(y^*,y^*)$ is a cPNE of $\mathcal{G}(G', \seq[i\in I]{g_i,Y_i} )$.

\subsubsection{As a generalized Nash game}
\label{sec:gen_ref}
When the feasible sets of the players' decision problems depend on each other's actions, we get a generalized Nash game with the corresponding solution called a generalized Nash equilibrium \citep{facchineiEA10}.
For our present purpose, we assume that there are $m+1$ players, indexed by $i \in \set{0,1,\dots,m}$, and player $i$'s decision problem is
\begin{equation}
\notag
\begin{aligned}
T_i(y_{-i}) = \arg
\min_{y_i}\; & h_i(y_{-i}, y_i) \\
\st
& (y_{-i}, y_i) \in H_i,
\end{aligned}
\end{equation}
where 
$H_i$ is a subset of
$\group{\prod_{j \neq i} \mbb{R}^{n_j}} \times \mbb{R}^{n_i}$
and
$h_i$ is a real-valued function on $H_i$.
We denote this game $\hat{\mathcal{G}}(\seq[i\in\set{0,1,\dots,m}]{h_i,H_i})$.
A point 
$(y_0^*, \dots, y_m^*)$
is a generalized Nash equilibrium (GNE) of $\hat{\mathcal{G}}(\seq[i\in\set{0,1,\dots,m}]{h_i,H_i})$ if 
$
y_i^* \in T_i(y_{-i}^*),
$
for all $i \in \set{0,1,\dots,m}$.

Given the game $\mathcal{G}(G, \seq[i\in I]{g_i,Y_i})$, we convert it to the generalized case by defining the $0^{th}$ player to have constraints that enforce $(x,y) \in G$ and a trivial objective:
define $h_0 : (y_{-0}, y_0) \mapsto 0$ and $H_0 = \set{(y_{-0}, y_0) : (y_0, y_1, \dots, y_m) \in G}$;
define for $i \in I$, $h_i : (y_{-i}, y_i) \mapsto g_i(y_0, y_i)$ and $H_i = \set{(y_{-i},y_i) : y_i \in Y_i}$.
Then again, it is simple to see that $(x^*, y^*)$ is a cPNE of the game $\mathcal{G}(G, \seq[i\in I]{g_i,Y_i})$ if and only if it is a GNE of $\hat{\mathcal{G}}(\seq[i\in\set{0,1,\dots,m}]{h_i,H_i})$.

Thus, the focus of our work can be considered a special case of GNE.
However, like the previously reviewed literature on pure Nash equilibrium, the theory and methods for generalized Nash equilibrium are also largely limited to the case that the players' decision problems are convex optimization problems.
Again, see \cite{facchineiEA10}.
Consequently, that literature does not help in the cases that we will consider.


\subsection{Terminology}
Moving forward, we will focus on the problem of finding a cPNE according to \Cref{defn:cpne}, and all subsequent problems will be posed in the form given in \Cref{sec:defn}.
Consequently, we may simply call a point an \emph{equilibrium} if it satisfies \Cref{defn:cpne} for an appropriately defined game.

\section{Characterization of equilibrium and minimum disequilibrium}
\label{sec:characterization}
In this section we pose an abstract optimization problem that can be used to characterize an equilibrium, and as a by-product, we introduce the concept of minimum disequilibrium.
We also motivate this concept by discussing connections with unit commitment in electricity markets.

\subsection{An optimization reformulation}
We can express the problem of finding a cPNE as a feasibility problem for the following bilevel program:
\begin{align}
\notag
\inf_{x,y_1,\dots,y_m}\; &0 \\
\st
\notag &(x,y_1,\dots,y_m) \in G, \\
\notag &y_i \in \arg\min_{z_i} \set{ g_i(x,z_i) : z_i \in Y_i }, \quad \forall i \in I.
\end{align}
Here, and throughout the rest of the work, we use the variables $z_i$ as an alias for the $i^{th}$ player's decision variables $y_i$ when, for instance, they appear in a ``lower-level'' problem.
This is common notation in the bilevel programming literature.
We can reformulate this bilevel problem as
\begin{subequations}
\label{pne_sip}
\begin{align}
\inf_{x,y_1,\dots,y_m}\; &0 \\
\st
& (x,y_1,\dots,y_m) \in G, \\
& y_i \in Y_i, \quad\forall i \in I, \label{pne_sip:feas1} \\
& g_i(x,y_i) \le g_i(x,z_i), \quad\forall z_i \in Y_i, \quad\forall i \in I.\label{pne_sip:infinite}
\end{align}
\end{subequations}
Such a reformulation has been considered by, for instance, \cite{bard83,tsoukalasEA09},
although it is easy to see that $y_i$ is optimal for the $i^{th}$ lower-level player problem if and only if it is feasible (Constraint~\eqref{pne_sip:feas1}) and its objective value is less than that of any other feasible point (Constraint~\eqref{pne_sip:infinite}).
Depending on the cardinality of the $Y_i$ sets, Problem~\eqref{pne_sip} may have an infinite number of constraints
(for each $i$, $z_i$ is effectively an index for the last set of constraints \eqref{pne_sip:infinite}).
Thus, \eqref{pne_sip} may be categorized as a semi-infinite program (SIP) 
(see e.g. \cite{guerravazquezEA08,stein12} for reviews).
Whether or not \eqref{pne_sip} is truly infinitely constrained depends on the nature of the $Y_i$ sets;
however, whether or not it is infinitely constrained does not complicate the following analysis and discussion.
Thus, to allow for the most general context possible, we will refer to problem~\eqref{pne_sip} as an SIP.

Once again, \eqref{pne_sip} is a feasibility problem;
evidently the most difficult constraint to handle is the ``infinite'' one \eqref{pne_sip:infinite}.
Thus, we can consider replacing the trivial objective with one which aims to minimize the violation of the infinite constraint.
To this end, let 
\[
\begin{aligned}
\allg   : (x,y) \mapsto &(g_1(x,y_1), \dots, g_m(x,y_m) ), \\
\allg^* : (x)   \mapsto &(g^*_1(x), \dots, g^*_m(x) ).
\end{aligned}
\]
When a player problem is unbounded, $\allg^*$ is extended real-valued.
Define
$\bar{\mbb{R}} = \mbb{R} \cup \set{-\infty,+\infty}$
(with the usual order), their Cartesian product
$\bar{\mbb{R}}^m = \bar{\mbb{R}} \times \dots \times \bar{\mbb{R}}$,
and the non-negative orthant
$\bar{\mbb{R}}^m_+ = \set{ w \in \bar{\mbb{R}}^m : w_i \ge 0, \forall i \in I }$.

In the following result, which characterizes equilibrium as a certain solution of an optimization problem, the objective can be thought of as a measure of ``disequilibrium,'' or roughly the dissatisfaction of all the players in aggregate.
This depends on an $\bar{\mbb{R}}$-valued function $\mu$;
a convenient choice is
\[
	\mu : (w_1,\dots,w_m) \mapsto \smallsum_{i=1}^m w_i,
\]
although we allow for different forms
(for instance, any norm satisfies the required properties).

\begin{proposition}
\label{prop:equilibrium}
Let $\mu : \bar{\mbb{R}}^m \to \bar{\mbb{R}}$ be any function satisfying:
\begin{enumerate}\itemsep0pt \parskip0pt
\item if $w \in \bar{\mbb{R}}^m_+$ then $\mu(w) \ge 0$;
\item $w \in \bar{\mbb{R}}^m_+$ and $\mu(w) = 0$ if and only if $w_i = 0$ for all $i \in I$.
\end{enumerate}
Consider
\begin{equation}
\label{md}
\tag{$\mathcal{MD}$}
\begin{aligned}
\delta \defn
\inf_{x,y_1,\dots,y_m}\;
&\mu(\allg(x,y) - \allg^*(x)) \\
\st 
& (x, y_1,\dots,y_m) \in G, \\
& y_i \in Y_i, \quad\forall i \in I.
\end{aligned}
\end{equation}
Any $(x^*,y_1^*,\dots, y_m^*)$ is a solution of \eqref{md} with $\delta = 0$ if and only if it is a cPNE of the game $\mathcal{G}(G, \seq[i\in I]{g_i,Y_i})$.
\end{proposition}
\begin{proof}
First, note that for any $(x,y_1,\dots,y_m)$ feasible in \eqref{md}, we must have $g_i(x,y_i) \ge g_i^*(x)$ (since $y_i$ is feasible in \eqref{player_i}).
Thus, since $\mu(w) \ge 0$ for all $w \ge 0$, the objective value of \eqref{md} is bounded below by zero;
that is, $\delta \ge 0$.

Let $(x^*,y_1^*,\dots,y_m^*)$ be an equilibrium.
Then $(x^*,y_1^*,\dots,y_m^*) \in G$ and for each $i$, $y_i^*$ is optimal for \eqref{player_i};
thus $g_i(x^*,y_i^*) = g_i^*(x^*)$ for each $i$, and so $\mu(\allg(x^*,y^*) - \allg^*(x^*)) = 0$.
Furthermore, $y_i^*$ is feasible in \eqref{player_i} so $y_i^* \in Y_i$ for each $i$;
therefore, $(x^*,y_1^*,\dots,y_m^*)$ is feasible in \eqref{md}, and has an objective value equal to zero.
It follows that $\delta \le 0$, but using the reverse inequality established above, we have $\delta = 0$, and $(x^*,y_1^*,\dots,y_m^*)$ is a solution of \eqref{md}.

Let $(x^*,y_1^*,\dots,y_m^*)$ be a solution of \eqref{md} with $\delta = 0$.
As noted above, by feasibility of $(x^*,y^*)$, we have 
$\allg(x^*,y^*)  - \allg^*(x^*) \ge 0$.
But since $\mu(\allg(x^*,y^*)  - \allg^*(x^*)) = \delta = 0$, using the properties of $\mu$, we must have for all $i$ that
$g_i(x^*,y_i^*) = g_i^*(x^*)$;
but further, $y_i^*$ is feasible for \eqref{player_i} and thus optimal.
Finally, since we also have $(x^*,y_1^*,\dots,y_m^*) \in G$, it follows that $(x^*,y_1^*,\dots,y_m^*)$ is an equilibrium.
\end{proof}

Characterization of an equilibrium in terms of the solution of an optimization problem has been proposed before, typically utilizing the NI function.
See \Cref{app:ni_function} for further discussion on this connection as well as an alternate proof of Proposition~\ref{prop:equilibrium} from this perspective.
However, we reiterate that the practical use of an optimization problem such as \eqref{md} to characterize equilibrium has, until this work, been confined to cases assuming a certain amount of convexity.

\subsection{Minimum disequilibrium and application in electricity markets}
We can interpret a solution of Problem~\eqref{md} as the ``closest'' point to being an equilibrium.
Specifically, its solution is a point $(x^*,y_1^*,\dots,y_m^*) \in G$ such that $y_i^*$ is $\epsilon_i$-optimal in \eqref{player_i} for each $i$,
that is,
\[
g_i(x^*,y_i^*) = g_i^*(x^*) + \epsilon_i,
\]
and $\mu( \epsilon_1,\dots, \epsilon_m )$ is at a minimum.
This perspective suggests that finding a solution with $\delta > 0$ may still be useful and provide a meaningful or practically useful point.
In general, we will refer to a solution of Problem~\eqref{md} as a \emph{minimum disequilibrium} solution, whether or not the solution is in fact an equilibrium.

As mentioned, minimum disequilibrium agrees with the minimum total opportunity cost idea in \cite{fullerEA17}.
Thus, that work provides excellent motivation for the value of a minimum disequilibrium solution for decision making in certain electricity markets;
we provide a brief discussion here.
In some electricity markets, a market operator has the problem of setting short-term electricity prices.
The players in this setting are electricity generators who decide their level of generation given prices, and the global linking constraints $G$ may model a required minimum level of electricity generation to meet demand, as well as limits on prices.
Generators may have fixed start-up costs;
this introduces nonconvexities, and a competitive equilibrium may not exist.
To address this, different approaches can be taken.
For instance, in the PJM
electricity market ``uplift'' or ``make whole'' payments are introduced, to ensure that generators operate at the required levels to meet demand without losing money.
Naturally, this leads to the idea of minimizing the total of the uplift payments.

A related formulation, the minimum total opportunity cost, aims to find prices ($x$), and for each generator $i$, operational levels and other decisions ($y_i$), so that the total discrepancy between the generators' realized profits ($-g_i(x,y_i)$) and what they \emph{could} have made ($-g_i^*(x)$) is minimized.
The result is a formulation that fits into the form of Problem~\eqref{md}.
We will explore this further in \Cref{sec:uc}.

\section{Solution method for minimum disequilibrium}
\label{sec:solution_methods}

Problem~\eqref{md} may be nontrivial to solve.
Evidently, the challenge stems from the fact that the players' optimal value functions, $g_i^*$, are in general implicitly defined.
If the player problems have some known structure that permits a more algebraic description of the optimal value function, then that could be used to reformulate the problems.
Notably, if the player problems are convex, arguments from duality theory or KKT conditions can be used to obtain various mathematical programs; see for instance \cite{gabriel_etal}, or from the perspective of SIP, \cite{steinEA03}.

However, our aim in this section is to develop methods that can solve Problem~\eqref{md} (or an equivalent reformulation) under fairly broad assumptions, including the cases that the player problems are nonconvex.
In particular, since verification of equilibrium requires the global solution of \eqref{md}, our discussion is motivated by methods for global optimization.

For the sake of clarity, we will assume that the function $\mu$ appearing in Proposition~\ref{prop:equilibrium} takes the previously mentioned form $\mu : w \mapsto \sum_i w_i$.
Thus, for each $i$ we can introduce a new scalar variable $w_i$ which essentially approximates $g_i^*$ from below, and Problem~\eqref{md} becomes
\begin{align}
\label{md_ref}
\tag{$\mathcal{MD}'$}
\delta = 
\inf_{x,y,w}	& \smallsum_{i\in I} (g_i(x,y_i)  - w_i) \\
\st 				
\notag & (x,y_1,\dots,y_m) \in G, \\
\notag & y_i \in Y_i,		\quad \forall i \in I, \\
\notag & w_i \le g_i^*(x), 	\quad \forall i \in I.
\end{align}

\begin{algorithm}
\caption{Solution method for Problem~\eqref{md_ref}}
\label{alg:CG}
\begin{algorithmic}[1]
\REQUIRE
$\varepsilon>0$,
$Y_i^{L,0} \subset Y_i$, 
$Y_i^{L,0} \neq \emptyset$, for each $i \in I$,
non-negative sequences
$\seq[k\in\mbb{N}]{\epsilon^k}$
and
$\seq[k\in\mbb{N}]{\eta_i^k}$, for each $i \in I$
\STATE
	$\delta^{U,0} = +\infty$,
	$\delta^{L,0} = 0$
\FOR{$k \in \mbb{N}$}
	\STATE
		\textbf{Solve Lower Bounding Problem}
		\begin{align}
		\label{md_ref_LB}
		\munderbar\delta^{k} \defn 
		\inf_{x,y,w}	& \smallsum_{i\in I} (g_i(x,y_i)  - w_i) \\
		\st
		\notag & (x,y_1,\dots,y_m) \in G, \\
		\notag & y_i \in Y_i,			\quad \forall i \in I, \\
		\notag & w_i \le g_i(x,z_i), 	\quad \forall z_i \in Y_i^{L,k}, \quad\forall i \in I,
		\end{align}
		approximately to obtain feasible solution $(x^k,y^k,w^k)$
		with objective value 
		\[\bar\delta^{L,k} = \smallsum_{i\in I} (g_i(x^k,y_i^k)  - w_i^k)\]
		and
		lower bound $\munderbar\delta^{L,k} \le \munderbar\delta^k$ such that
		$\abs{\bar\delta^{L,k} - \munderbar\delta^{L,k}} \le \epsilon^k$
	\STATE
		\textbf{Update lower bound:} $\delta^{L,k} \gets \max\set{\munderbar\delta^{L,k}, \delta^{L,k-1}}$
	\STATE
		\textbf{Solve player problems} \eqref{player_i} for each $i \in I$, for $x = x^k$
		\begin{equation}
		\notag
		g_i^*(x^k) = \inf_{z_i} \set{g_i(x^k, z_i) : z_i \in Y_i},
		\end{equation}
		approximately to obtain feasible solution $z_i^k$
		and lower bound $g_i^{L,k} \le g_i^*(x^k)$, such that
		$\abs{g_i(x^k, z_i^k) - g_i^{L,k}} \le \eta_i^k$.
	
	\IF{$\epsilon^k \le \varepsilon$ and $w_i^k \le g_i^{L,k}$, for all $i \in I$}
	\label{step:early_term}
		\STATE $(x^*,y^*) \gets (x^k,y^k)$
		\RETURN $(x^*,y^*)$
	\ENDIF
	\STATE
	$Y_i^{L,k+1} \gets Y_i^{L,k} \cup \set{z_i^k}$ for each $i \in I$
	\STATE
	\textbf{Update upper bound:}\\
	$\bar\delta^{U,k} \gets \smallsum_{i\in I} (g_i(x^k,y_i^k) - g_i^{L,k})$
	\STATE
	$\delta^{U,k} \gets \min\set{\bar\delta^{U,k}, \delta^{U,k-1}}$
	\IF{$\bar\delta^{U,k} < \delta^{U,k-1}$}
		\STATE
			$(x^*,y^*) \gets (x^k,y^k)$
	\ENDIF
	\IF{$\delta^{U,k} - \delta^{L,k} \le \varepsilon$}
	\label{step:term}
		\RETURN $(x^*,y^*)$
	\ENDIF
\ENDFOR
\end{algorithmic}
\end{algorithm}

The overall solution method is given in \Cref{alg:CG}.
The structure to note is that in the lower bounding problem, Problem~\eqref{md_ref_LB}, at any iteration $k$, the set $Y_i^{L,k}$ is a subset of $Y_i$ (for each $i$).
In other words, $Y_i^{L,k}$ is a subset of the feasible set of the $i^{th}$ player, and so we have that
$g_i^*(x) \le \inf\set{ g_i(x, z_i) : z_i \in Y_i^{L,k} }$.
Thus, Problem~\eqref{md_ref_LB} is a relaxation of Problem~\eqref{md_ref}, and so $\munderbar\delta^{k}$ is a lower bound on $\delta$, as desired.
By solving the player problems, we add constraints to the lower bounding problem and improve the lower bound.
This general structure is also shared by the algorithm from \cite{blankenshipEA76} for the solution of SIP.
Similar extensions in the context of SIP are explored in \cite{djelassi2017hybrid,harwoodEA16,turan2023optimality}.

However, the specialization of this class of SIP solution methods to finding equilibria is novel and unexplored.
Furthermore, in our general setting, both the lower bounding problem and the player subproblems can be nonconvex MINLPs, which are generally difficult to solve exactly (i.e., with a zero percent optimality gap), especially for large-scale instances.
Consequently, practitioners may wish to solve these instances to provable approximate optimality within a given non-zero tolerance, e.g., $0.1\%$, and use the primal and dual bounds instead of a more difficult-to-obtain optimal objective function value with $0\%$ optimality gap.
To this end, \Cref{alg:CG} allows for approximate solution of the various subproblems, making it relevant as a method for guaranteeing an approximate minimum disequilibrium when faced with challenging MINLP subproblems.

We remark on a few details of the algorithm.
The initial subsets $Y_i^{L,0}$ can be arbitrary;
they should merely be nonempty so that Problem~\eqref{md_ref_LB} is not unbounded on the first iteration.
In practice, initially each $Y_i^{L,0}$ set should be finite as well
(otherwise we have missed the point of the algorithm, which is to avoid infinite constraints).
Single points are added to them through the course of the algorithm and so they remain finite.
Further, since $0$ is an {\it a priori} lower bound on $\delta$, the lower bound is initialized as $\delta^{L,0} = 0$.
Meanwhile, \Cref{alg:CG} can terminate in one of two ways.
Either:
\begin{enumerate}
\item
At Step~\ref{step:early_term}, the approximate solution of Problem~\eqref{md_ref_LB} is feasible, and thus approximately optimal, for Problem~\eqref{md_ref}.
Specifically, if $w_i^k \le g_i^{L,k}$ for all $i$, then $(x^k, y^k, w^k)$ is feasible in Problem~\eqref{md_ref}.
Since $(x^k, y^k, w^k)$ is feasible, its objective value $\bar\delta^{L,k}$ must be greater than $\delta$.
Then we have
$\munderbar{\delta}^k \le \delta \le \bar\delta^{L,k} \le \munderbar{\delta}^k + \epsilon^k$.
Thus, if the lower bounding problem is solved accurately enough ($\epsilon^k \le \varepsilon$), then $(x^k, y^k, w^k)$ is $\varepsilon$-optimal in Problem~\eqref{md_ref}.

If one only cares about the existence of an equilibrium solution, the algorithm can terminate early if it happens that $\delta^{L,k} > 0$;
since we have that $\delta^{L,k}$ is a lower bound on $\delta$, we have by Proposition~\ref{prop:equilibrium} that an equilibrium solution does not exist.
\item
An $\varepsilon$-optimal solution $(x^*,y^*)$ of Problem~\eqref{md_ref} is found.
Note that
$\sum_i (g_i(x,y_i) - g_i^*(x))$
is an upper bound for $\delta$;
we have merely evaluated the objective at a feasible point of Problem~\eqref{md_ref}.
More specifically, at any iteration for which the upper bound is finite, we have a point $(x^*,y^*) = (x^{k'},y^{k'})$ which came from a solution of Problem~\eqref{md_ref_LB} (at a previous iteration $k'$);
thus, $(x^{k'},y^{k'}) \in G$ and $y_i^{k'} \in Y_i$ for each $i$ and so 
$(x,y,w) = (x^{k'},y^{k'},g^{L,k'})$
is feasible in Problem~\eqref{md_ref} since $g_i^{L,k'} \le g_i^*(x^{k'})$ for each $i$ by construction.
Consequently, its objective value is an upper bound on $\delta$, and the upper bound $\delta^{U,k}$ tracks the best of these upper bounds.
\end{enumerate}

The following result establishes when the algorithm produces an $\varepsilon$-optimal solution in finite iterations.
To be clear, we obtain a point $(x^*,y^*) \in G$ satisfying
$y_i^* \in Y_i$ for each $i$, and 
$\sum_i \group{g_i(x^*,y_i^*) - g_i^*(x^*)} \le \delta + \varepsilon$.
We also obtain upper and lower bounds on $\delta$, and if the upper bound is small, then as in the discussion following Proposition~\ref{prop:equilibrium}, we have a point that is ``close'' to being an equilibrium.
In general, we obtain an $\varepsilon$-approximate minimum disequilibrium solution.
See \Cref{app:termination_proof} for its proof.

\begin{theorem}
\label{thm:cg_termination}
Assume that the set
$\set{ (x,y) : (x,y_1,\dots,y_m) \in G, y_i \in Y_i, \forall i \in I}$
is compact and nonempty.
Assume that for each $i$, $g_i$ is continuous and $Y_i$ is compact.
Let
$\epsilon^* = \limsup_{k \to \infty} \epsilon^k$
and
$\eta_i^* = \limsup_{k\to\infty} \eta_i^k$ for each $i$.
Then for any
$\varepsilon > \epsilon^* + 2 \smallsum_{i \in I} \eta_i^*$,
\Cref{alg:CG} produces an $\varepsilon$-optimal solution $(x^*,y^*)$ of Problem~\eqref{md_ref} in finite iterations.
\end{theorem}

Given appropriate algebraic descriptions of the $G$ and $Y_i$ sets and $g_i$ functions, implementation of \Cref{alg:CG} in an algebraic modeling language like AIMMS%
\footnote{Copyright \copyright{} 2021 AIMMS B.V. All rights reserved. AIMMS is a registered trademark of AIMMS B.V. \url{www.aimms.com}}
or GAMS \citep{gams} is a relatively simple matter.
These optimization modeling environments provide access to high-quality numerical methods for the global solution of the player subproblems~\eqref{player_i} and lower-bounding subproblem~\eqref{md_ref_LB} required by the algorithm.
In some practically relevant situations, these subproblems are mixed-integer linear or quadratic programs (MILP/MIQP), which typically can be solved robustly and at scale, despite being nonconvex problems.
See also the example in \Cref{sec:uc}.

\section{Primal-dual optimality and minimum disequilibrium}
\label{sec:pc_duality}

\Cref{alg:CG} is a method for solving the minimum disequilibrium problem \eqref{md_ref} under very general assumptions.
In this section, we take advantage of specific problem structure to decompose the corresponding minimum disequilibrium problem into a primal problem and its Lagrangian dual problem.
The problem structure that we analyze is consistent with spatial price equilibrium (SPE) problems and other models of competitive behavior.
The SPE problem is a classic problem going back to \cite{samuelson1952spatial}, although see also \cite{nagurney2022spatial} for a recent application and further references.
In the SPE problem, spatially distributed producers and consumers participate in a competitive market for a commodity;
the goal is to find the trade flows between the participants and prices at equilibrium.
When the players are modeled by convex optimization problems, Samuelson's basic approach \cite{samuelson1952spatial} establishes that optimal dual variables can be interpreted as equilibrium prices for SPE problems.
We will see that our primal-dual decomposition of the minimum disequilibrium problem provides a generalization, to the nonconvex case, of this classic result.
See Theorem~\ref{thm:dual_disequilibrium}.

First, assume that each player's optimization problem has the form
\[
\inf_{z_i} \set{ g_i^a(z_i) + \sum_{k \in K} x_k g_{i,k}^b(z_i) : z_i \in Y_i },
\]
for nonempty $Y_i \subset \mbb{R}^{n_i}$, appropriate $\mbb{R}$-valued functions $g_i^a$ and $g_{i,k}^b$, and where
$K \defn \set{1,\dots,n_0}$.
That is, their feasible sets are constant/independent of $x$, and their objectives have related forms.
Once again, denote the overall objective of this player problem as $g_i(x,z_i)$ and the optimal objective value as $g_i^*(x)$.
Further, assume that the global constraint set $G$ has the following form:
\[
G = 
	\mbb{R}^{n_0} \times 
	\set{ 
	(y_1, \dots, y_m)
		: \smallsum_{i \in I} g_{i,k}^b(y_i) = 0, \forall k \in K 
	}.
\]

This particular structure is relevant for modeling price-taking behavior of players in multiple market structures.
As a simple example, consider $m=2$ players:
player 1 produces a commodity, while player 2 consumes it.
If $x_1$ is the price of the commodity and $g_{1,1}^b(y_1) = -y_1$ models the negative quantity that the producer sells, then $x_1 g_{1,1}^b(y_1)$ is negative revenue, and with the term $g_1^a$ modeling the cost of production, this models a profit-maximizing producer.
Meanwhile, $g_{2,1}^b(y_2) = y_2$ would model how much the consumer purchases, and at equilibrium we would require that what is sold equals what is purchased, $y_1 = y_2$, or $g_{1,1}^b(y_1) + g_{2,1}^b(y_2) = 0$.
This is precisely what the global constraints $G$ enforce.
The example considered in \Cref{sec:spe} also fits into this form.

Finally, to avoid some pathological edge cases, we will need the following assumption that there exists a value of the prices $x$ such that each player problem is (simultaneously) bounded.
This helps in Theorem~\ref{thm:dual_disequilibrium} below.
\begin{assumption}
\label{assm:bounded}
There exists $x' \in \mbb{R}^{n_0}$ such that
$g_i^*(x') > -\infty$ for each $i \in I$.
\end{assumption}

The significance of these assumptions becomes clear when specializing problem~\eqref{md} under these assumptions.
If we let $\mu(w) = \smallsum_i w_i$, what is important to note is that for all $(x,y_1,\dots,y_m) \in G$, we have
\[
\smallsum_{i \in I} g_i(x,y_i) = \smallsum_{i \in I} g_i^a(y_i),
\]
since the global constraints set $G$ includes the conditions
$\smallsum_i g_{i,k}^b(y_i) = 0$, for each $k$,
which cause the other terms in the objectives cancel.
Thus we obtain
\begin{align}
\label{md_pc}
\delta = 
\inf_{x,y_1,\dots,y_m}\; & \sum_{i \in I} g_i^a(y_i) - \sum_{i \in I} g_i^*(x) \\
\st
\notag & \smallsum_{i \in I} g_{i,k}^b(y_i) = 0, \quad\forall k \in K, \\
\notag & x \in \mbb{R}^{n_0}, \\
\notag & y_i \in Y_i, \quad\forall i \in I.
\end{align}
Note that under Assumption~\ref{assm:bounded}, it follows that $\delta$ must be finite when Problem~\eqref{md_pc} is feasible;
take any value of $y$ feasible in \eqref{md_pc} and we see
$0 \le \delta \le \sum_i g_i^a(y_i) - \sum_i g_i^*(x') < +\infty$.

Inspecting problem \eqref{md_pc}, we see that its feasible set is 
$\mbb{R}^{n_0} \times \set{ y \in \prod_{i} Y_i : \smallsum_{i} g_{i,k}^b(y_i) = 0, \forall k \in K }$.
Thus we can write 
$\delta = \delta^P - \delta^{{D}}$ 
where
\begin{align}
\label{primal}
\tag{$\mathcal{P}$}
\delta^P \defn
\inf_{y_1,\dots,y_m}\; & \smallsum_{i \in I} g_i^a(y_i)  \\
\st
\notag & \smallsum_{i \in I} g_{i,k}^b(y_i) = 0, \quad\forall k \in K, \\
\notag & y_i \in Y_i, \quad\forall i \in I,
\end{align}
and
\begin{align}
\notag
\delta^{D} &\defn
-\inf_{x \in \mbb{R}^{n_0}}\; -\smallsum_{i \in I} g_i^*(x) \\
\label{dual_like}
\tag{$\mathcal{D}$}
&\begin{aligned}
=\sup_{x\in \mbb{R}^{n_0}} 
	&\inf_{z}\; 
	\smallsum_{i \in I} g_i^a(z_i) + \smallsum_{k \in K} x_k \smallsum_{i \in I} g_{i,k}^b(z_i)\\
	&\st z_i \in Y_i, \quad\forall i \in I,
\end{aligned}
\end{align}
where we have used the definition of $g_i^*$ and the fact that $\sup_x f(x) = -\inf_x(-f(x))$ for any real function $f$.

We note that \eqref{dual_like} is in fact the Lagrangian dual problem of \eqref{primal}.
This observation inspires us to characterize a minimum disequilibrium solution in terms of optimal primal-dual solutions.
The following result formalizes this and states that optimal dual variables $x^*$ are the prices corresponding to a minimum disequilibrium solution, and the duality gap equals the minimized disequilibrium.
When an equilibrium exists, we regain the classic interpretation of optimal dual variables as equilibrium prices.

\begin{theorem}
\label{thm:dual_disequilibrium}
Suppose Assumption~\ref{assm:bounded} holds.
If $(x^*,y^*)$ is a minimum disequilibrium solution of \eqref{md_pc}, then $\delta$ equals the duality gap $\delta^P - \delta^D$.
Further, $(x^*,y^*)$ is a minimum disequilibrium solution of \eqref{md_pc}
if and only if
$y^* = (y_1^*,\dots, y_m^*)$ is optimal for \eqref{primal} and 
$x^*$ is optimal for \eqref{dual_like}.
\end{theorem}
\begin{proof}
%
%
First we show the equivalence of solutions of \eqref{md_pc} and solutions of \eqref{primal} and \eqref{dual_like}.
Let $y^*$ and $x^*$ be optimal for \eqref{primal} and \eqref{dual_like}, respectively.
Then we must have that $(x^*,y^*)$ is feasible for \eqref{md_pc}.
For a contradiction, assume that $(x^*,y^*)$ is not optimal for \eqref{md_pc}.
Then there exists $(x,y)$ feasible in \eqref{md_pc} with
$\sum_i g_i^a(y_i)   - \sum_i g_i^*(x) < 
	\sum_i g_i^a(y_i^*) - \sum_i g_i^*(x^*)$.
This implies that either 
$- \sum_i g_i^*(x) < - \sum_i g_i^*(x^*)$,
or
$\sum_i g_i^a(y_i) < \sum_i g_i^a(y_i^*)$,
or both;
in any case this contradicts the optimality of $x^*$, $y^*$, or both.
Thus $(x^*,y^*)$ is optimal for \eqref{md_pc}
(that is, $(x^*,y^*)$ is a minimum disequilibrium solution).

Conversely, assume that $(x^*,y^*)$ is optimal for \eqref{md_pc}.
As noted, under Assumption~\ref{assm:bounded}, 
$\delta = \sum_i g_i^a(y_i^*) - \sum_i g_i^*(x^*)$
must be finite.
Since $g_i^a$ is real-valued, 
$\sum_i g_i^*(x^*) = \delta - \sum_i g_i^a(y_i^*)$ is also finite.
First, assume for a contradiction that $x^*$ is not optimal in \eqref{dual_like}.
Then there exists $x$ such that 
\[
	- \smallsum_i g_i^*(x) < - \smallsum_i g_i^*(x^*).
\]
However, since the objective function of the player is real-valued and $Y_i$ is nonempty for each $i$, $g_i^*(x)$ cannot be $+\infty$-valued, and so 
$- \sum_i g_i^*(x)$
is finite.
Adding $\sum_i g_i^a(y_i^*)$ to either side of the above inequality yields
$\sum_i g_i^a(y_i^*) - \sum_i g_i^*(x) < 
	\sum_i g_i^a(y_i^*) - \sum_i g_i^*(x^*)$,
which contradicts the optimality of $(x^*, y^*)$.
Next, for a contradiction assume that $y^*$ is not optimal in \eqref{primal}.
Then there exists $y$ feasible in \eqref{primal} such that
$\sum_i g_i^a(y_i) < \sum_i g_i^a(y_i^*)$.
Subtracting $\sum_i g_i^*(x^*)$ from either side of this inequality gives
$\sum_i g_i^a(y_i) - \sum_i g_i^*(x^*) < 
	\sum_i g_i^a(y_i^*) - \sum_i g_i^*(x^*)$,
which again contradicts the optimality of $(x^*, y^*)$.
Thus $y^*$ and $x^*$ are optimal for \eqref{primal} and \eqref{dual_like}, respectively.

To show the first claim, let $(x^*,y^*)$ be a minimum disequilibrium solution of \eqref{md_pc}.
Then as already shown, $y^*$ and $x^*$ are optimal for \eqref{primal} and \eqref{dual_like}.
In this case 
$\delta^P = \sum_i g_i^a(y_i^*)$
and
$\delta^{D} = \sum_i g_i^*(x^*)$
and so
$\delta^P - \delta^{D} = \sum_i g_i^a(y_i^*) - \sum_i g_i^*(x^*) = \delta$.
\end{proof}

As an example of when Theorem~\ref{thm:dual_disequilibrium} reduces to the classic case, if Problem~\eqref{primal} is feasible and bounded, and for each $i$, $g_i^a$ is convex, $g_{i,k}^b$ is linear for each $k$, and $Y_i$ is polyhedral, then standard Lagrangian duality results imply that strong duality holds
(see for instance \cite[Prop.~5.2.1]{bertsekas_nlp}).
Then Theorem~\ref{thm:dual_disequilibrium} implies that the optimal solution of \eqref{primal} along with the optimal Lagrange multipliers of the dualized constraints yield an equilibrium solution.

However, we reiterate that Theorem~\ref{thm:dual_disequilibrium} holds more generally.
Regardless of convexity, Theorem~\ref{thm:dual_disequilibrium} establishes that if we solve the primal problem and its dual, we will arrive at a minimum disequilibrium solution that may have value in certain situations or as an approximate equilibrium.

Finally, Theorem~\ref{thm:dual_disequilibrium} is a rather clean version of various results in the literature dealing with pricing in electricity markets.
For instance, Theorem 8 of \cite{fuller2008} addresses a different problem form and asserts that a point is an equilibrium if and only if it is primal and dual optimal and there is zero duality gap.
Certain results about the properties of convex hull pricing, introduced by \cite{gribikEA07} and studied further by \cite{schiro2015convex}, go through similar arguments.
The application of the present work to that literature is a fertile area of future research.

\subsubsection*{Specialization of \Cref{alg:CG}}
\Cref{alg:CG} also applies in this setting.
We briefly discuss how \Cref{alg:CG} specializes to a method for solving the dual problem.
By similar reasoning that we can decompose the minimum disequilibrium problem~\eqref{md_pc}, the lower-bounding problem~\eqref{md_ref_LB} in this setting decomposes into the primal problem~\eqref{primal} above and
\begin{equation}
\label{relaxed_dual}
\begin{aligned}
\delta^{D,U} = 
\sup_{x,w}\; & \sum_i w_i \\
\st
& w_i \le g_i^a(z_i) + \smallsum_{k \in K} x_k g_{i,k}^b(z_i),
 	\quad\forall z_i \in Y_i^L, \quad\forall i \in I.
\end{aligned}
\end{equation}
This is a relaxation of the dual problem~\eqref{dual_like}, and thus $\delta^{D,U}$ is an upper bound on $\delta^D$.
The primal problem is static and only needs to be solved once;
\Cref{alg:CG} reduces to iteratively solving the relaxed dual~\eqref{relaxed_dual} and the player problems.
Solving the player problems is equivalent to minimizing the Lagrangian of \eqref{primal}:
\begin{equation}
\label{min_lagrangian}
\begin{aligned}
\sum_i g_i^*(x) =
\inf_{z_1,\dots,z_m}\; 
&\sum_i \group{ g_i^a(z_i) + \smallsum_{k \in K} x_k g_{i,k}^b(z_i) } \\
\st
& z_i \in Y_i, \quad\forall i \in I.
\end{aligned}
\end{equation}
We recognize this overall procedure as a bundle subgradient method or cutting plane method for maximizing the concave dual function.
See for instance \cite[Section~6.3.3]{bertsekas_nlp}.
In order to apply Theorem~\ref{thm:cg_termination}, we would need to assume that each $Y_i$ is compact, the defining functions $g_i^a$ and $g_{i,k}^b$ are all continuous, and additionally impose bounds on the dual variables.
See \Cref{sec:spe} for an application of this procedure.

\section{Examples}
\label{sec:ex}
Here we present some examples in order to illustrate applications of our theoretical and numerical developments.
Our numerical experiments, and in particular applications of \Cref{alg:CG}, are implemented in AIMMS version 4.72.
For more extensive numerical tests in the setting of ``competitive pooling,'' we refer to the authors' recent work in \citep{papageorgiou2023pooling}.
That work includes examples involving players modeled by mixed-integer quadratically-constrained nonconvex quadratic programs, with up to 60 continuous variables and 5 binary variables each.

\subsection{Example: Discretely-constrained Cournot players}
\label{sec:dcc}
We consider a simple two-player example from \cite[Section 2.1]{papageorgiou2021note}.
This example was originally introduced to illustrate that KKT-based methods for equilibrium are ill-suited to the situation of discretely-constrained players.
In particular, the method from \cite{gabriel2013B} may fail to yield an equilibrium even if one exists.
We draw the same conclusion, and in addition demonstrate that \Cref{alg:CG} finds the equilibrium.

In their simplest form, each player $i$, for $i \in \set{1,2}$, is modeled by the optimization problem
\[
	\inf_{y_i} \set{-y_1 - y_2 : y_i \in [0, 1.1] \cap \mbb{Z}}
\]
where $\mbb{Z}$ is the set of integers.
The feasible set of each player's decision problem is just $\set{0,1}$, and the unique equilibrium is $(y_1,y_2) = (1,1)$.

Ignoring the discrete constraints (that is, taking the continuous relaxation) of each player yields a convex optimization problem.
This inspires the approach of \cite{gabriel2013B}, for instance.
In this approach, the KKT conditions of each player's continuous relaxation are combined together, and then the integrality conditions are re-introduced.
The result is that one seeks the solution of the following discretely-constrained linear complementarity problem:
\begin{subequations}
\label{eq:dclcp}
\begin{align}
\label{eq:dclcp:one}
0 \le \lambda_i - 1 \perp y_i \ge 0, 	&\quad\forall i \in \set{1,2}, \\
\label{eq:dclcp:two}
0 \le 1.1 - y_i \perp \lambda_i \ge 0,	&\quad\forall i \in \set{1,2}, \\
y_i \in \set{0,1},						&\quad\forall i \in \set{1,2}.
\end{align}
\end{subequations}
We note, however, that $(y_1,y_2) = (1,1)$ does \emph{not} satisfy these conditions:
Conditions~\eqref{eq:dclcp:one} imply we must have $\lambda_i = 1$ for each $i$,
while Conditions~\eqref{eq:dclcp:two} imply $\lambda_i = 0$ for each $i$.
Thus, we have a contradiction.
In fact, it is easy to verify that Conditions~\eqref{eq:dclcp} have no solution.

Let us consider how \Cref{alg:CG} would apply to this problem.
First, we reformulate the game into the appropriate format via the discussion in \Cref{sec:common_ref}.
We introduce the global constraint set 
\[
	G = \set{ (x_1,x_2,y_1,y_2) \in \mbb{R}^4 : x_1 = y_1, x_2 = y_2}
\]
and redefine the players' problems as
\[\begin{aligned}
	g_1^*(x) &= \inf_{y_1} \set{ -y_1 - x_2 : y_1 \in [0, 1.1] \cap \mbb{Z}}, \\
	g_2^*(x) &= \inf_{y_2} \set{ -y_2 - x_1 : y_2 \in [0, 1.1] \cap \mbb{Z}}.
\end{aligned}\]
To initialize \Cref{alg:CG}, we set $Y_1^L = Y_2^L = \set{0}$.
Then the lower-bounding problem~\eqref{md_ref_LB} is initially
\[\begin{aligned}
\delta^L = 
\inf_{x,y,w} & (-y_1 - x_2) + (-y_2 - x_1)  - w_1 - w_2 \\
\st
& x_1 = y_1, \\
& x_2 = y_2, \\
& y_1 \in [0, 1.1] \cap \mbb{Z}, \\
& y_2 \in [0, 1.1] \cap \mbb{Z}, \\
& w_1 \le 0 - x_2, \\
& w_2 \le 0 - x_1.
\end{aligned}\]
We note that at every iteration, the lower-bounding problem is an MILP, as are the player problems.
Furthermore, even though this example suffers from a poor formulation (the continuous relaxations of the players are weak), we do not have to worry about that directly;
we can leave this issue to robust and constantly-improving solvers for MILP.

However, for this particular example we can solve the required problems by inspection.
The solution of the lower-bounding problem on the first iteration is
$(x_1,x_2,y_1,y_2,w_1,w_2) = (1,1,1,1,-1,-1)$
with optimal objective value $\delta^L = -2$.
Solving the player problems for $x = (1,1)$ we get $z_1 = z_2 = 1$, with optimal objective values $g_1^*(x) = g_2^*(x) = -2$.
Adding these points to the $Y_i^L$ sets, we get $Y_1^L = Y_2^L = \set{0,1}$.
However, we get an upper bound on the disequilibrium of 
$\delta^U = (-y_1 - x_2) + (-y_2 - x_1) - g_1^*(x) - g_2^*(x) = (-2) + (-2) - (-2) - (-2) = 0$.
Thus, the termination condition at Step~\ref{step:term} of \Cref{alg:CG} implies that 
$x_i = y_i = 1$ for all $i$ is the equilibrium solution, as desired.

\subsection{Application: Unit commitment in electricity markets}
\label{sec:uc}
As discussed at the end of \Cref{sec:characterization}, models of electricity markets are a source of equilibrium problems with nonconvex players.
Here, we present an example from \cite{fullerEA17} to demonstrate that \Cref{alg:CG} can handle these practical problems.
We also note that no other method can solve this example.
The players are mixed-integer quadratic programs (MIQP) and the global/side constraints are nontrivial;
the methods listed in \Cref{tab:methods} (besides the proposed method) either do not apply in this case, or else have no guarantees and give ambiguous answers.
Further, the nontrivial side constraints complicate solving this example as a potential game%
\footnote{As previously noted, potential games with continuous potentials and compact strategy sets always have a pure Nash equilibrium;
see for instance \cite[Lemma 2.1]{monderer1996potential}.
We will see that this example in fact does not have an equilibrium, which is another indication that this example cannot be formulated as a potential game with a well-behaved potential function.
}.

The example is the ``single period unit commitment model'' from \cite{fullerEA17}.
We have three price-taking, profit-maximizing power producers modeled by MIQPs.
An integer (specifically binary) decision is required to model a fixed cost for starting up generation.
Thus, the producer players are parametric in the price $x^p$;
for $i \in I = \set{1,2,3}$, we have
\begin{subequations}
\begin{align}
\inf_{y_i = (y_i^b, y_i^c)}\; &a_i y_i^c + (\sfrac{1}{2}) c_i (y_i^c)^2 + b_i y_i^b - x^p y_i^c \\
\st
& y_i^b C_i^L \le y_i^c \le y_i^b C_i^U, \\
& y_i^b \in \set{0,1}, \quad y_i^c \in \mbb{R}.
\end{align}
\end{subequations}
Data are in \Cref{tab:uc_data}.

\begin{table}
\caption{Data for producer players of unit commitment example of \Cref{sec:uc}}
\label{tab:uc_data}
\centering
\begin{tabular}{cccccc}
\hline
Index $i$	&$a_i$	&$c_i$	 &$b_i$		&$C_i^L$ &$C_i^U$ \\
\hline
$1$			&$10$	&$0.05$	 &$4000$	&$400$	 &$600$ \\
$2$			&$45$	&$0.1$	 &$100$		&$200$	 &$250$ \\
$3$			&$35$	&$0.002$ &$2000$	&$300$	 &$500$ \\
\hline
\end{tabular}
\end{table}

In the terminology of \cite{fullerEA17}, the consumption side of this problem is ``nondispatchable.''
In other words, consumption should \emph{not} be treated as a player potentially contributing to disequilibrium.
Consumption is modeled instead by an inverse demand curve that provides constraints between the level of consumption $x^q$ and price $x^p$ that are included in the global constraints $G$.
Thus we have $x = (x^p, x^q)$ and
\[
	G = \set{ (x, y_1, y_2, y_3) : 
		x^p = \alpha - \beta x^q,
		x^q = \smallsum_{i=1}^3 y_i^c,
		x^q \in [0, C_1^U + C_2^U + C_3^U]},
\]
where $\alpha = 200$ and $\beta = 0.2$.
Note that $G$ also includes the basic requirement that production and consumption balance, as well as bounds on consumption implied by production bounds.

The ultimate goal is to determine price and consumption that minimize disequilibrium, which in this case equals total opportunity cost defined by \cite{fullerEA17}.
\Cref{alg:CG} is applicable.
We first guess a consumption level of $x^q = 100$, set the price accordingly
(i.e. $x^p = \alpha - \beta x^q$),
solve the player problems, and use the optimal solutions to initialize the $Y_i^L$ sets.
This initial guess of consumption is arbitrary and does not affect the finite termination of \Cref{alg:CG}.
We can then solve the lower-bounding problem~\eqref{md_ref_LB} which takes the form
\begin{align}
\notag
\delta^L = 
\inf_{x,y,w}\; & \smallsum_{i \in I} \group{a_i y_i^c + (\sfrac{1}{2}) c_i (y_i^c)^2 + b_i y_i^b - x^p y_i^c  - w_i} \\
\st
\notag & x^p = \alpha - \beta x^q, \\
\notag & x^q = \smallsum_i y_i^c, \\
\notag & x^q \in [0, C_1^U + C_2^U + C_3^U], \\
\notag & y_i^b C_i^L \le y_i^c \le y_i^b C_i^U, \quad \forall i \in I,\\
\notag & y_i^c \in \mbb{R}, \quad \forall i \in I,\\
\notag & y_i^b \in \set{0,1}, \quad \forall i \in I, \\
\notag & w_i \le a_i z_i^c + (\sfrac{1}{2}) c_i (z_i^c)^2 + b_i z_i^b - x^p z_i^c,  \quad \forall z_i \in Y_i^L, \quad\forall i\in I.
\end{align}
Note that we may use the constraint $x^p = \alpha - \beta x^q$ to eliminate the variable $x^p$.
Further, we can transform the objective to 
\[
	-(\alpha - \beta x^q)x^q + 
	\smallsum_i (a_i y_i^c + (\sfrac{1}{2}) c_i (y_i^c)^2 + b_i y_i^b - w_i)
\]
where we have used the constraint $x^q = \smallsum_i y_i^c$.
Thus, the lower-bounding problem may be solved as an MIQP.
We use CPLEX to solve the player problems and lower-bounding problems.

\begin{table}
\caption{Minimum disequilibrium solution of unit commitment example of \Cref{sec:uc}}
\label{tab:uc_soln}
\centering
\begin{tabular}{ccccc}
\hline
$x^{p,*}$
		& $x^{q,*}$
				& $(y_1^{b,*}, y_1^{c,*})$
						& $(y_2^{b,*}, y_2^{c,*})$
								& $(y_3^{b,*}, y_3^{c,*})$
\\
\hline
$39.5$	& $802.5$
				& $(1,502.5)$
						& $(0,0)$
								& $(1,300)$ \\
\hline
\end{tabular}
\end{table}

The method converges after three iterations when we have $\delta^L = \delta^U = 931.41$
(where, for simplicity, we solve the subproblems exactly).
The minimum disequilibrium price and consumption are $x^{p,*} = 39.5$ and $x^{q,*} = 802.5$;
see \Cref{tab:uc_soln} for the full solution.
These values happen to agree with those found by a heuristic procedure in \cite[Appendix D]{fullerEA17}.
Meanwhile, \citet{fullerEA17} also report the minimum complementarity solution from \cite{gabriel2013A}.
They find that the disequilibrium/total opportunity cost achieved by the minimum complementarity method is $988$ (see \cite[Table 2]{fullerEA17}).
The issue with the heuristic procedure of \cite{fullerEA17} and the minimum complementarity approach is that there is no guarantee that a minimum disequilibrium solution is found;
we may be left wondering whether an equilibrium does in fact exist.
In contrast, thanks to Theorem~\ref{thm:cg_termination}, we can guarantee that $931.41$ is in fact the minimum disequilibrium value, and that no equilibrium exists.

	

\subsection{Application: Price equilibrium in a natural gas network}
\label{sec:spe}
We present an example in spatial price equilibrium related to a gas transmission network and apply the primal-dual approach of \Cref{sec:pc_duality}.
This example is modified from the Belgian natural gas system model of \cite{chen2019equilibria}.
The original example was an optimization problem over continuous variables with nonconvex constraints, which the authors approached by taking a convex relaxation of the problem and subsequently formulating the KKT conditions.
We will instead retain the original nonconvex constraints, and further modify the example by adding integer decisions.
Yet, by application of Theorem~\ref{thm:dual_disequilibrium}, we can establish that an equilibrium exists and determine the equilibrium prices.

In the basic setting of the problem, we have a network of nodes (cities) connected by gas pipelines, and at each node there may be either supplies of gas entering the network, or consumption of gas from the network, or possibly neither.
See \Cref{fig:network}.
Loosely, we assume we have a gas market with a single price at each node, and we wish to find the equilibrium prices at each node.

\begin{figure}
\centering
\tikzstyle{supply}=[regular polygon,draw,regular polygon sides=4]
\tikzstyle{transm}=[regular polygon,draw,regular polygon sides=6]
\tikzstyle{demand}=[circle,draw]
\tikzstyle{widearrow}=[-{Latex[width=2mm]}]
\begin{tikzpicture}[xscale=1.5,every node/.style={font=\tiny}]
\node[supply]  (1) at (-2,2) {1};
\node[supply]  (2) at (-1,2) {2};
\node[demand]  (3) at (0,2) {3};
\node[transm]  (4) at (1,2) {4};
\node[demand]  (7) at (2,2) {7};
\node[demand]  (6) at (3,2) {6};
\node[supply]  (5) at (4,2) {5};
\node[supply]  (8) at (5,2) {8};
\node[transm]  (9) at (5,0) {9};
\node[demand] (10) at (4,0) {10};
\node[transm] (11) at (3,0) {11};
\node[demand] (12) at (2,0) {12};
\node[supply] (13) at (1,0) {13};
\node[supply] (14) at (0,0) {14};
\node[demand] (15) at (-1,0) {15};
\node[demand] (16) at (-2,0) {16};
\node[transm] (17) at (3,-2) {17};
\node[transm] (18) at (2,-2) {18};
\node[demand] (19) at (1,-2) {19};
\node[demand] (20) at (0,-2) {20};
\draw[widearrow] (1) -- (2);
\draw[widearrow] (3) -- (4);
\draw[widearrow] (4) -- (7);
\draw[widearrow] (2) -- (3);
\draw[widearrow] (4) -- (14);
\draw[widearrow] (5) -- (6);
\draw[widearrow] (7) -- (6);
\draw[widearrow] (8) -- (9);
\draw[widearrow] (9) -- (10);
\draw[widearrow] (10) -- (11);
\draw[widearrow] (11) -- (12);
\draw[widearrow] (11) -- (17);
\draw[widearrow] (12) -- (13);
\draw[widearrow] (13) -- (14);
\draw[widearrow] (14) -- (15);
\draw[widearrow] (15) -- (16);
\draw[widearrow] (17) -- (18);
\draw[widearrow] (18) -- (19);
\draw[widearrow] (19) -- (20);
\end{tikzpicture}
\caption{Schematic of gas transmission network for the price equilibrium example of \Cref{sec:spe}.
Square nodes are the locations of supplies, while circular nodes are the locations of demands
(hexagonal nodes have neither supply nor demand).
Arcs show connecting pipelines, with arrows denoting direction of flow.
See \cite[Fig.~3]{chen2019equilibria} for a geographic interpretation of the network.}
\label{fig:network}
\end{figure}
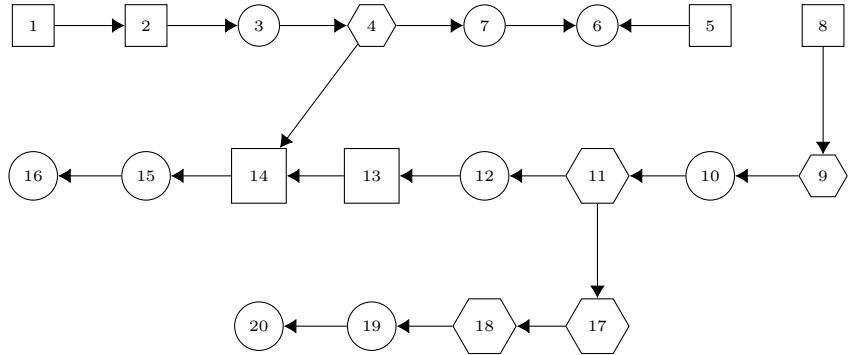

To formally model this problem in our framework, let the nodes of the network be $K$, and the arc set be $A$.
We assume that we have a price-taking transmission player operating on the whole network, buying and selling gas at each node market.
At each node we may additionally have a price-taking supply player who sells gas to the market, or a price-taking demand player who purchases gas from the market.
Let the set of nodes with a supply player be $K^s \subset K$;
let the set of nodes with a demand player be $K^d \subset K$.

Thus, given each price at node $k$, $x_k$, the transmission player maximizes their profits while obeying the physical constraints between flow and pressure:
\begin{subequations}
\label{gas:transmission}
\begin{align}
\inf_{y^f, y^p}\; & \sum_{(k,k') \in A} (x_k - x_{k'}) y_{kk'}^f \\
\st
\label{gas:transmission:start}
& y_k^p \in \mbb{R},
	\quad y_k^p \in [b_k^p, B_k^p], \quad \forall k \in K,\\
& y_{kk'}^f \in \mbb{R},
	\quad 0 \le y_{kk'}^f, \quad \forall (k,k') \in A, \\
\label{gas:transmission:noncovexity}
& y_{k}^p - y_{k'}^p = \group{\sfrac{y_{kk'}^f}{w_{kk'}}}^2, \quad \forall (k,k') \in A.
\end{align}
\end{subequations}
The decision variables are $y_{kk'}^f$ for each arc $(k,k')$, the flow rate of gas in the pipeline between nodes $k$ and $k'$, and $y_k^p$ for each node $k$, the squared pressure of the gas at the node.
Meanwhile, $b_k^p$ and $B_k^p$ are lower and upper bounds, respectively, on the squared pressure at the nodes, and $w_{kk'}$ is a parameter of the pipeline between the nodes.
Each term in the objective equals the negative profit from buying the gas at price $x_k$ at node $k$, transporting it to node $k'$, and selling it at price $x_{k'}$.
The objective equals the negative total profit summed over all pipelines/arcs.
It is important to note that Constraints~\eqref{gas:transmission:noncovexity} are physical constraints relating flow to the pressure drop in the pipeline;
these make the transmission player nonconvex.

Meanwhile, given the price at node $k$, $x_k$, a supply player at node $k \in K^s$ chooses the flow rate of gas $y_k^s$ to sell to the market in order to maximize their profits:
\begin{subequations}
\label{gas:supply}
\begin{align}
\inf_{y_k^s}\; & (c_k^s - x_k) y_k^s \\
\st
\label{gas:supply:start}
& y_k^s \in \mbb{R}, \quad  y_k^s \in [0, B_k^s].
\end{align}
\end{subequations}
Here, $c_k^s$ is the cost of supply, while $B_k^s$ is an upper bound on the supply capacity.

We assume that the demand players are electric power producers, and as a consequence have fixed costs associated with start up.
Thus, given the price at the relevant node, $x_k$, a demand player at node $k \in K^d$ chooses whether to start up, and if so chooses the flow rate of gas $y_k^d$ to purchase in order to maximize profits:
\begin{subequations}
\label{gas:demand}
\begin{align}
\inf_{y_k^d, y_k^b}\; & (x_k - c_k^d) y_k^d + c_k^b y_k^b\\
\st
\label{gas:demand:start}
& 0 \le y_k^d \le y_k^b B_k^d, \\
\label{gas:demand:stop}
& y_k^b \in \set{0,1}, \quad y_k^d \in \mbb{R}.
\end{align}
\end{subequations}
Here, $c_k^d$ is the marginal utility of the demand, $c_k^b$ is the fixed cost, and $B_k^d$ is an upper bound on the demand flow rate.
For both the supply and demand players, the objectives equal negative profits.

At equilibrium, we expect that for each node, the total flow into a node equals the total flow out of the node.
So, define for each $k \in K$
\[
g^b_k(y^s, y^d, y^f) \defn \chi_k^d y_k^d + \sum_{k': (k,k')\in A} y_{kk'}^f - \chi_k^s y_k^s - \sum_{k': (k',k)\in A} y_{k'k}^f,
\]
where $\chi_k^d = 1$ if $k \in K^d$ and $0$ otherwise, and similarly for $\chi_k^s$.
Then, we have the global constraints
$G = \mbb{R}^{\card{K}} \times \set{ (y^s, y^d, y^b, y^f, y^p) : g^b_k(y^s, y^d, y^f) = 0, \forall k \in K}$.


Our goal is to follow the analysis in \Cref{sec:pc_duality} and form the corresponding primal problem for this example.
To start, we first rewrite the transmission player objective as
\[\begin{aligned}
\sum_{(k,k') \in A} (x_k - x_{k'}) y_{kk'}^f 
&= 	\sum_{(k,k') \in A} x_k y_{kk'}^f - \sum_{(k,k') \in A} x_{k'} y_{kk'}^f \\
&= 	\sum_{(k,k') \in A} x_k y_{kk'}^f - \sum_{(k',k) \in A} x_{k}  y_{k'k}^f \\
&=	\sum_{k \in K}\group{\sum_{k':(k,k') \in A} x_k y_{kk'}^f}  - 
	\sum_{k \in K}\group{\sum_{k':(k',k) \in A} x_{k} y_{k'k}^f}.
\end{aligned}\]
In the second line, we have just reindexed the arcs in the second sum;
in the third line, we have rearranged the first sum over arcs by grouping the terms by ``outgoing'' arcs over all nodes, while we rearrange the second sum over arcs by grouping terms by ``incoming'' arcs over all nodes.

With this expression for the transmission player objective, it is easy to see that the sum of all the players' objectives is
\begin{multline*}
\sum_{k \in K^s} (c_k^s - x_k) y_k^s +
\sum_{k \in K^d} \group{(x_k - c_k^d) y_k^d + c_k^b y_k^b} +
\sum_{(k,k') \in A} (x_k - x_{k'}) y_{kk'}^f
= \\
\sum_{k \in K^s} c_k^s y_k^s + \sum_{k \in K^d} \group{-c_k^d y_k^d + c_k^b y_k^b} + 
\sum_{k\in K} x_k  g^b_k(y^s, y^d, y^f)
\end{multline*}
Of course, for any $(x, y^s, y^d, y^b, y^f, y^p) \in G$, the third sum in this expression is exactly zero.
Thus, as in \Cref{sec:pc_duality}, the minimum disequilibrium problem can be decomposed into the primal problem
\begin{align}
\notag
\delta^P = 
\inf_{y^s, y^d, y^b, y^f, y^p}\; &\sum_{k \in K^s} c_k^s y_k^s + \sum_{k \in K^d} \group{-c_k^d y_k^d + c_k^b y_k^b}\\
\st
\notag
& \text{\eqref{gas:transmission:start} - \eqref{gas:transmission:noncovexity}}, \\
\notag
& \text{\eqref{gas:supply:start}}, \quad \forall k \in K^s, \\
\notag
& \text{\eqref{gas:demand:start} - \eqref{gas:demand:stop}}, \quad \forall k \in K^d, \\
\label{gas:primal:dualized}\tag{$**$}
& 0 = g^b_k(y^s, y^d, y^f),
	\quad \forall k \in K,
\end{align}
and its corresponding Lagrangian dual problem (with Constraints~\eqref{gas:primal:dualized} being dualized).
We highlight that this is a mixed-integer nonlinear program (MINLP) with a nonconvex continuous relaxation.

Although the primal problem is nonconvex, Theorem~\ref{thm:dual_disequilibrium} still applies.
The player problems all have compact feasible sets and continuous objectives and so Assumption~\ref{assm:bounded} holds;
consequently Theorem~\ref{thm:dual_disequilibrium} asserts that the minimum disequilibrium equals the duality gap $\delta^P - \delta^D$, where we recall that $\delta^D$ is the optimal objective value of the dual problem.

Using the numerical values of the parameters in \Cref{app:gas_data}, we solve the primal problem to global optimality with BARON version 19 \citep{baron,tawarmalaniEA05}.
The optimal primal objective value is $-101.060$.
We fix the binary variables to their optimal value and re-solve the primal problem locally to get Lagrange multipliers.
These multipliers provide an intelligent guess for optimal dual variables, but without further effort they remain only a guess:
we must evaluate the dual function at these values, and if there is still a duality gap then we must resort to something like the cutting plane method discussed at the end of \Cref{sec:pc_duality}.
However, again using BARON to minimize the Lagrangian (to global optimality) at this guess for the dual variables, we see that the dual function achieves the optimal primal objective value, and we verify that strong duality indeed holds.
For completeness, we test the performance of the cutting plane method discussed at the end of \Cref{sec:pc_duality};
initializing each dual variable to zero, the method achieves an absolute duality gap of less than $10^{-3}$ after 266 iterations, which takes about 40 seconds on a laptop computer
(specifically, a four-core Intel i5-8350U CPU, with a nominal clock speed of 1.70 GHz).
In any case, we have an equilibrium solution;
the solution of the dual problem yields prices such that the players have no incentive to deviate from the solution given by the primal problem.
See \Cref{app:gas_data} for the full solution.


\section{Conclusions}
This work has analyzed and expanded upon the connections between pure Nash equilibrium and bilevel programming.
Some of these connections seem to exist in the literature in vague terms;
by posing them independently of the typical settings of complementarity problems, variational inequalities, or the Nikaido-Isoda function, we have been able to leverage ideas from the broader literature on nonconvex bilevel programs in order to propose solution algorithms for equilibrium problems with potentially nonconvex players.
The connections with bilevel programming and SIP provide directions for future research in applying solution methods for those problems to general equilibrium problems.

The concept of minimum disequilibrium was introduced as an alternative solution when no equilibrium exists.
It was shown with an example that this solution concept can handle the minimum total opportunity cost idea in unit commitment problems.
Furthermore, special but still economically relevant situations (spatial price equilibrium) can be analyzed to show that optimal dual variables have significance as part of a minimum disequilibrium solution.
In general, the examples have demonstrated the flexibility of our definitions and how specialization of \Cref{alg:CG} to various settings provides a principled way to find practical solution methods to nonconvex equilibrium problems.

\subsubsection*{Acknowledgments}
The authors would like to thank their colleague Nicolas Sawaya for introducing some of the challenges associated with equilibrium modeling within a real-world setting, as well as the various fruitful discussions on this topic over the past several years.
The authors would also like to thank their colleagues Myun-Seok Cheon and Youngdae Kim for similarly fruitful discussions while developing this work.

\section*{Declarations}
{\bf Conflict of Interest.}
The authors have no relevant financial or non-financial interests to disclose.

\noindent
{\bf Data availability.} The authors did not analyse or generate any datasets, because the work proceeds within a theoretical and mathematical approach.

\appendix
\section{Connections with the Nikaido-Isoda function}
\label{app:ni_function}

Optimization formulations of equilibrium problems have been presented in the literature before.
Many rely on the Nikaido-Isoda (NI) function, first proposed by \citet{nikaidoEA55};
see also \cite{facchineiEA10,vonheusingerEA09} for recent treatments and generalizations.
We give an alternate proof of Proposition~\ref{prop:equilibrium} using the NI function approach.

Consider the game $\hat{\mathcal{G}}(\seq[i\in\set{0,1,\dots,m}]{h_i,H_i})$ from \Cref{sec:gen_ref}.
For $y = (y_0, y_1, \dots, y_m)$, define
\[
\begin{aligned}
\phi(y) \defn 
\sup_{z_0, z_1, \dots, z_m}\;
&\sum_{i=0}^{m} \group{ h_i(y_{-i},y_i) - h_i(y_{-i},z_i) } \\
\st
& (y_{-i},z_i) \in H_i, \forall i \in \set{0,1,\dots,m}.
\end{aligned}
\]
The objective function in this optimization problem defining $\phi$ is the NI function as defined by \cite{facchineiEA10,vonheusingerEA09}.
Theorem~3.2 of \cite{facchineiEA10} states that
$y^* = (y_0^*, y_1^*, \dots, y_m^*)$ 
is a GNE of $\hat{\mathcal{G}}(\seq[i\in\set{0,1,\dots,m}]{h_i,H_i})$ if and only if 
$\phi(y^*) = 0$ and 
\begin{equation}
\label{eq:ni_optimization}
y^* \in \arg\min_y\set{\phi(y) : (y_{-i}, y_i) \in H_i, i \in \set{0,1,\dots,m}}.
\end{equation}

To use this characterization of equilibrium, consider the problem of finding a cPNE of $\mathcal{G}(G,\seq[i\in I]{g_i,Y_i})$, and following the discussion in \Cref{sec:gen_ref} we obtain an equivalent game $\hat{\mathcal{G}}(\seq[i\in\set{0,1,\dots,m}]{h_i,H_i})$ 
where
$H_0 \defn \set{(y_{-0}, y_0) : (y_0, y_1, \dots, y_m) \in G}$,
$H_i \defn \set{(y_{-i}, y_i) : y_i \in Y_i}$,
$h_i(y_{-i}, y_i) \defn g_i(y_0, y_i)$,
and $h_0$ is identically zero.
With this, we can express the feasible set of \eqref{eq:ni_optimization} in terms of the set $G$ and data of the \eqref{player_i} problems as
\[
\begin{aligned}
\Phi \defn
\big\{ 
& (y_0,y_1, \dots, y_m) : \\
&   (y_0, y_1, \dots, y_m) \in G, \\
&   y_i \in Y_i, \forall i \in \set{1,\dots,m}
\big\},
\end{aligned}
\]
which coincides with the feasible set of Problem~\eqref{md} 
(defining $x \defn y_0$).
Using the definition of the $h_i$ (and the fact that $h_0$ is identically zero), we transform $\phi$ into
\begin{equation*}
\begin{aligned}
\phi(y) = 
\smallsum_{i=1}^{m} g_i(y_0,y_i) - 
\inf_z \big\{ &\smallsum_{i=1}^{m} g_i(y_0,z_i) : \\
    & (z_0, y_1, \dots, y_m) \in G,\\
    & z_i \in Y_i,  \forall i \in \set{1,\dots,m}
\big\}.
\end{aligned}
\end{equation*}
For $y \in \Phi$, the infimum above is over a nonempty set, and furthermore, does not depend on the $z_0$ variable.
Consequently, we can ignore the side constraints encoded in $G$ and decompose the minimization.
Thus, for $y \in \Phi$, the expression for $\phi(y)$ simplifies to 
\[
\phi(y) = 
\smallsum_{i=1}^{m} g_i(y_0,y_i) - \smallsum_{i=1}^{m} g_i^*(y_0)
\]
where we recall the optimal player value function $g_i^*$ defined in Equation~\eqref{eq:optimal_value}.
Finally, note that this expression equals the objective function of Problem~\eqref{md} when 
$\mu : w \mapsto \smallsum_i w_i$.
Thus, Problems~\eqref{md} and \eqref{eq:ni_optimization} coincide, and the statement that $y^*$ is an equilibrium iff $\phi(y^*) = 0$ and $y^* \in \arg\min\set{\phi(y) : y \in \Phi}$, essentially provides an alternate proof of Proposition~\ref{prop:equilibrium}.

\section{Proof of Theorem~\ref{thm:cg_termination}}
\label{app:termination_proof}
One effect of solving the player problems inexactly is that the lower bounds (at worst) converge to the optimal value of the relaxation 
\begin{align}
\label{md_ref_relaxation}
\tilde\delta =
\inf_{x,y,w}	& \smallsum_{i\in I} (g_i(x,y_i)  - w_i) \\
\st 				
\notag & (x,y_1,\dots,y_m) \in G, \\
\notag & y_i \in Y_i,		\quad \forall i \in I, \\
\notag & w_i \le g_i^*(x) + \eta_i^*, 	\quad \forall i \in I.
\end{align}
By defining $w_i' = w_i - \eta_i^*$ and re-writing the objective as
$\smallsum_{i\in I} (g_i(x,y_i)  - w_i' - \eta_i^*)$
we see that
\[
\tilde\delta = \delta - \smallsum_i \eta_i^*.
\]
Similarly, the upper bounds will only reach within $\sum_i \eta_i^*$ of $\delta$.
The following proof makes this precise.

\begin{theorem*}[\bf\ref{thm:cg_termination}]
Assume that the set
\[
    \nashset \equiv \set{ (x,y) : (x,y_1,\dots,y_m) \in G, y_i \in Y_i,\forall i \in I }
\]
is compact and nonempty.
Assume that for each $i$, $g_i$ is continuous and $Y_i$ is compact.
Let
$\epsilon^* = \limsup_{k \to \infty} \epsilon^k$
and
$\eta_i^* = \limsup_{k\to\infty} \eta_i^k$ for each $i$.
Then for any
$\varepsilon > \epsilon^* + 2 \smallsum_{i \in I} \eta_i^*$,
\Cref{alg:CG} produces an $\varepsilon$-optimal solution $(x^*,y^*)$ of Problem~\eqref{md_ref} in finite iterations.
\end{theorem*}
\begin{proof}
We establish that the upper and lower bounds converge to some values $\delta^{U,*}$ and $\delta^{L,*}$, respectively, such that
$\delta^{L,*} \le \delta \le \delta^{U,*}$
and
$\delta^{U,*} - \delta^{L,*} \le \epsilon^* + 2\smallsum_i \eta_i^*$.

To begin, we show the the approximate solutions of the lower bounding problem~\eqref{md_ref_LB} have a subsequence that converge to a feasible point of the relaxation~\eqref{md_ref_relaxation}.
Let 
$\seq[k \in \mbb{N}]{(x^k,y^k,w^k)}$
be the sequence of feasible solutions of Problem~\eqref{md_ref_LB} produced by \Cref{alg:CG}.
Since more elements are added to $Y_i^{L,k}$ at each iteration, the part of the solution sequence $\seq[k]{w_i^k}$ is non-increasing, but bounded below by the minimum of $g_i$ on $\nashset$; 
continuity and compactness ensure that this is finite
(specifically,
$\inf_{x,y,z_i} \set{ g_i(x,z_i) : z_i \in Y_i, (x,y) \in \nashset } > -\infty$).
Thus, for each $i$, $\set{w_i^k : k \in \mbb{N}}$ is contained in a compact set, and so the entire solution sequence is in a compact set.
For each $i$, let $\seq[k \in \mbb{N}]{z_i^k}$ be the corresponding sequence of approximate solutions to the player problem~\eqref{player_i};
these must exist at each iteration $k$ by continuity of $g_i$ and compactness of $Y_i$.
Again, the image of this sequence $\set{ z_i^k : k \in \mbb{N} }$ is in a compact set ($Y_i$) for each $i$.
Consequently, we have that a subsequence of solutions converges to some point.
Abusing notation, we have that 
$(x^k,y^k,w^k) \to (x^*,y^*,w^*)$ and 
$z^k \to z^*$.
Note that we have $(x^*,y^*) \in \nashset$.

Now, we establish that $(x^*,y^*,w^*)$ is feasible in the relaxation~\eqref{md_ref_relaxation}.
Since $z_i^k$ is added to $Y_i^{L,k}$ at the end of each iteration, we have for each $i$
\[
 w_i^{\ell} \le g_i(x^{\ell}, z_i^k), \quad \forall \ell,k : \ell > k.
\]
By taking the limit over $\ell$, and then the limit over $k$, we get for each $i$
\begin{equation}
\label{eq:limit_feasibility}
w_i^* \le g_i(x^*, z_i^*).
\end{equation}
Now, for a contradiction, assume that for some $i$, $w_i^* > g_i^*(x^*) + \eta_i^*$,
indicating that $(x^*,y^*,w^*)$ is not feasible in Problem~\eqref{md_ref_relaxation}.
This means that there exists $z_i^{\dagger} \in Y_i$ 
(feasible in the player problem)
with 
\begin{equation}
\label{eq:for_contradiction}
w_i^* > g_i(x^*, z_i^{\dagger}) + \eta_i^*.
\end{equation}
By definition of $z_i^k$ as an approximate minimizer of \eqref{player_i} for $x = x^k$, we have
$g_i(x^k, z_i^{\dagger}) + \eta_i^k \ge g_i(x^k, z_i^{k})$
for all $k$, and taking the limit superior over $k$ we get
\[
g_i(x^*, z_i^{\dagger}) + \eta_i^* 
= \limsup_{k\to \infty}(g_i(x^k, z_i^{\dagger}) + \eta_i^k)
\ge \limsup_{k\to \infty} g_i(x^k, z_i^{k})
= g_i(x^*,z_i^*).
\]
Combined with Inequality~\eqref{eq:for_contradiction} this gives
\[
w_i^* > g_i(x^*,z_i^*),
\]
which contradicts \eqref{eq:limit_feasibility}.
Thus, $(x^*,y^*,w^*)$ is feasible in Problem~\eqref{md_ref_relaxation}, and in particular, satisfies for each $i \in I$
\begin{equation}
\label{eq:relaxed_feasibility}
w_i^* \le g_i^*(x^*) + \eta_i^*.
\end{equation}

Next, we focus on the lower bounds.
The algorithm's lower bound is constructed as
\[
\delta^{L,k} = \max\set{\munderbar\delta^{L,k}, \delta^{L,k-1}}
\]
and so forms a non-decreasing sequence.
From the approximate solution of the lower bounding problem~\eqref{md_ref_LB}, we have
$\munderbar\delta^{L,k} \le \munderbar\delta^k \le \delta$.
A simple induction argument establishes that $\delta^{L,k} \le \delta$ for all $k$, and so 
\[\begin{aligned}
\delta^{L,*} 
&\defn \lim_{k \to \infty} \delta^{L,k} \\
&\le \delta.
\end{aligned}\]
Since we have
$\munderbar\delta^{L,k} \le \delta^{L,k}$
for all $k$, it follows
$\bar\delta^{L,k} - \delta^{L,k} \le \bar\delta^{L,k} - \munderbar\delta^{L,k}$ for all $k$.
By construction, 
$\bar\delta^{L,k} - \munderbar\delta^{L,k} \le \epsilon^k$ for all $k$, 
and so we have
$\limsup_{k\to \infty} (\bar\delta^{L,k} - \delta^{L,k}) \le \limsup_{k\to \infty} \epsilon^k \defn \epsilon^*$.
Note that 
$\bar\delta^{L,k} = \smallsum_{i} (g_i(x^k,y_i^k)  - w_i^k)$
has a subsequential limit 
\[
\bar\delta^{L,*} \defn \smallsum_{i} (g_i(x^*,y_i^*)  - w_i^*).
\]
Combining this,
\begin{align}
\notag
\bar\delta^{L,*} - \delta^{L,*}
&\le
	\limsup_{k\to \infty} (\bar\delta^{L,k} - \delta^{L,k}) \\
\label{lower_bounds_bound}
&\le \epsilon^*.
\end{align}

Next, we focus on the upper bounds.
Note that for each $i$, $g_i^* : \mbb{R}^{n_0} \to \mbb{R}$ is a continuous function by \cite[Theorem~1.4.16]{aubin_frankowska}.
Thus,
\begin{equation}
\label{bar_delta}
\bar\delta^k \defn \smallsum_i (g_i(x^k,y_i^k) - g_i^*(x^k))
\end{equation}
has a subsequential limit 
$\smallsum_i (g_i(x^*,y_i^*) - g_i^*(x^*))$
and by Inequality~\eqref{eq:relaxed_feasibility}
\[\begin{aligned}
\bar\delta^* 
&\defn \smallsum_i (g_i(x^*,y_i^*) - g_i^*(x^*)) \\
&\le \smallsum_i (g_i(x^*,y_i^*) - w_i^* + \eta_i^*) \\
&= \bar\delta^{L,*} + \smallsum_i \eta_i^*.
\end{aligned}\]
For each $k$, $\bar\delta^k$ is an upper bound: $\bar\delta^k \ge \delta$.
Consequently, so is $\bar\delta^*$: $\bar\delta^* \ge \delta$.
Rearranging Inequality~\eqref{lower_bounds_bound}, we have
$\bar\delta^{L,*} - \epsilon^* \le \delta^{L,*}$.
Combining these relations, we get:
\begin{equation}
\label{eq:bounds_range}
\bar\delta^{L,*} - \epsilon^* \le \delta^{L,*}  \le \delta \le \bar\delta^* \le \bar\delta^{L,*} + \smallsum_i \eta_i^*.
\end{equation}
Thus, $\seq[k]{\bar\delta^k}$ converges to within $\epsilon^* + \smallsum_i \eta_i^*$ of both $\delta$ and the lower bound limit $\delta^{L,*}$.
It remains to show that the upper bounds that are actually calculated, $\delta^{U,k}$, also converge within a reasonable value.

To this end, note that 
\[
	\delta^{U,k} = \min\set{\bar\delta^{U,k}, \delta^{U,k-1}}
\]
where 
$\bar\delta^{U,k} = \smallsum_{i\in I} (g_i(x^k,y_i^k) - g_i^{L,k})$.
Combining this with Equation~\eqref{bar_delta}, we have
$\bar\delta^{U,k} - \bar\delta^k = \smallsum_i g_i^*(x^k) - g_i^{L,k}$.
By construction of $g_i^{L,k}$, 
$0 \le g_i^*(x^k) - g_i^{L,k} \le \eta_i^k$
and so
$0 \le \bar\delta^{U,k} - \bar\delta^k \le \smallsum_i \eta_i^k$.
It is simple to see that $\delta^{U,k} \ge \delta$ for all $k$, and that it is a non-increasing sequence, and so it must converge to some value greater than $\delta$:
\[\begin{aligned}
\delta^{U,*} 
	&\defn \lim_{k \to \infty} \delta^{U,k}\\
	&\ge \delta.
\end{aligned}\]
Further, $\delta^{U,k} \le \bar\delta^{U,k}$ for all $k$ and so 
$\delta^{U,k} - \bar\delta^k \le \smallsum_i \eta_i^k$.
Consequently, 
\[\begin{aligned}
\delta^{U,*} - \bar\delta^* 
&\le \limsup_{k\to \infty} (\delta^{U,k} - \bar\delta^k)\\
&\le \limsup_k \smallsum_i \eta_i^k\\
&\le \smallsum_i \limsup_k \eta_i^k\\
&= \smallsum_i \eta_i^*
\end{aligned}\]
Finally, using this and Inequality~\eqref{eq:bounds_range}, if $\delta^{U,*}$ is greater than $\bar\delta^*$, then
$\delta^{U,*} - \delta^{L,*} \le \epsilon^* + 2\smallsum_i \eta_i^*$.
Otherwise, we have
$\delta^{U,*} - \delta^{L,*} \le \epsilon^* + \smallsum_i \eta_i^*$
(since $\delta^{U,*} \ge \delta$).
In either case, we have the conclusion
\[
\delta^{U,*} - \delta^{L,*} \le \epsilon^* + 2\smallsum_i \eta_i^*.
\]
\end{proof}

\section{Data for and solution of gas network price equilibrium example}
\label{app:gas_data}
In this section we specify the data used for the example in \Cref{sec:spe}, as well as the equilibrium solution found by the primal-dual approach.
The topology of the network is in \Cref{fig:network}.
See Tables~\ref{tab:spe:nodes}, \ref{tab:spe:transmission}, \ref{tab:spe:supply}, and \ref{tab:spe:demand} for the data and solution of the overall nodes, pipelines/transmission, supplies, and demands, respectively.
In addition, we have lower and upper bounds on the squared pressure variable $b_k^p = 900$ (bar$^2$) and $B_k^p = 4900$ (bar$^2$), for each node $k \in K$.
Finally, based on the costs of supply and marginal utilities, we enforce the bounds $[0,12.1]$ for each price $x_k$ when solving the dual problem.
From \Cref{tab:spe:nodes}, we note that these bounds are not binding at the equilibrium solution.

\begin{table}[h]
\centering
\small
\begin{minipage}[t]{0.48\textwidth}
\caption{Price and squared pressure solution}
\label{tab:spe:nodes}
\begin{tabulary}{\textwidth}{CRR}
\hline
Node $k$
		& Equilibrium price, $x_k$ 
					& Equilibrium squared pressure, $y_k^p$ 
\\
\hline
1		&  5.500	& 3305.3 \\
2		&  5.500	& 3198.0 \\
3		&  5.500	& 3031.2 \\
4		&  5.500	& 2308.6 \\	
5		&  5.500	& 1818.8 \\
6		&  5.500	& 1818.8 \\
7		&  5.500	& 1928.3 \\
8		&  4.500 	& 4900.0 \\
9		&  4.507	& 4890.0 \\	
10		&  6.790 	& 1555.6 \\
11		&  6.804 	& 1548.8 \\
12		&  6.804 	& 1543.6 \\
13		&  5.500 	& 1543.6 \\	
14		&  5.500 	& 1543.6 \\
15		&  5.500 	& 1404.6 \\
16		&  5.500 	& 1236.5 \\
17		&  6.931	& 1529.2 \\
18		&  6.932	& 1529.1 \\
19		& 10.765	&  936.3 \\	
20		& 11.000	&  900.0 \\
\hline
\end{tabulary}
\end{minipage}
\quad
\begin{minipage}[t]{0.48\textwidth}
\caption{Pipeline data and solution}
\label{tab:spe:transmission}
\begin{tabulary}{\textwidth}{CCR}
\hline
Arc, $(k,k')$ \qquad\qquad\quad
			& Weymouth constant, $w_{kk'}$ (Mm$^3$/day/bar)
							& Equilibrium pipeline flow, $y_{kk'}^f$ 
\\
\hline
$(1,2)$ 	& 3.011688895 	& 31.200	\\
$(2,3)$ 	& 2.459034363 	& 31.754	\\
$(3,4)$ 	& 1.181283201 	& 31.754	\\
$(4,7)$ 	& 0.476334966 	&  9.290	\\
$(4,14)$	& 0.812192096 	& 22.464	\\
$(5,6)$ 	& 0.316632279 	&  0.000	\\
$(7,6)$ 	& 0.385558037 	&  4.034	\\
$(8,9)$ 	& 3.000000000  	&  9.490	\\
$(9,10)$	& 0.164342326 	&  9.490	\\
$(10,11)$	& 1.204674230  	&  3.125	\\
$(11,12)$	& 0.929427781 	&  2.120	\\
$(11,17)$	& 0.226813800  	&  1.005	\\
$(12,13)$	& 0.952379651 	&  0.000	\\
$(13,14)$	& 2.693737181 	&  0.000	\\
$(14,15)$	& 1.904759827 	& 22.464	\\
$(15,16)$	& 1.204674230  	& 15.616	\\
$(17,18)$	& 3.000000000  	&  1.005	\\
$(18,19)$	& 0.041270934 	&  1.005	\\
$(19,20)$	& 0.166790287 	&  1.005	\\
\hline
\end{tabulary}
\end{minipage}
\end{table}

\begin{table}
\caption{Supply bounds, costs data, and solution}
\label{tab:spe:supply}
\centering
\small
\begin{tabulary}{\textwidth}{CRRR}
\hline
Node index, $k$
	& Upper bound, $B_k^s$ (Mm$^3$/day)
			& Cost of supply, $c_k^s$ (1000\$/Mm$^3$/day)
					& Equilibrium supply flow, $y_k^s$ 
\\
\hline
1  	& 31.20	& 4.5 	& 31.200 \\
2  	&  8.40 & 5.5 	&  0.554 \\
5  	&  4.80 & 5.5 	&  0.000 \\
8  	& 20.00 & 4.5 	&  9.490 \\
13 	&  0.96	& 5.5 	&  0.000 \\
14 	&  1.20 & 5.5 	&  0.000 \\
\hline
\end{tabulary}
\end{table}

\begin{table}
\caption{Demand bounds, objective function data, and solution}
\label{tab:spe:demand}
\centering
\small
\begin{tabulary}{\textwidth}{CRRRR}
\hline
Node index, $k$
	& Upper bound, $B_k^d$ (Mm$^3$/day)
				& Marginal utility, $c_k^d$ (1000\$/Mm$^3$/day)
						& Fixed cost, $c_k^b$ (1000\$)
								& Equilibrium demand flow, $y_k^d$ 
\\
\hline
3  	&  3.918	& 7 	& 6		&  0.000 	\\
6  	&  4.034	& 7 	& 3		&  4.034 	\\
7  	&  5.256	& 7 	& 3		&  5.256 	\\
10 	&  6.365	& 8 	& 1		&  6.365 	\\
12 	&  2.120	& 8 	& 1		&  2.120 	\\
15 	&  6.848	& 7 	& 3 	&  6.848 	\\
16 	& 15.616	& 7 	& 3		& 15.616	\\
19 	&  0.222	& 11	& 1		&  0.000 	\\
20 	&  1.919	& 11	& 0		&  1.005 	\\
\hline
\end{tabulary}
\end{table}


\end{document}